\documentclass[11pt,reqno]{amsart}
\usepackage[utf8]{inputenc}

\usepackage[margin=1.2in, centering]{geometry}
\usepackage{amsmath,amssymb,latexsym,mathtools}

\usepackage{amsthm}
\newtheorem{thm}{Theorem}[section]
\newtheorem{lem}[thm]{Lemma}
\newtheorem{cor}[thm]{Corollary}
\newtheorem{prop}[thm]{Proposition}

\numberwithin{equation}{section}

\theoremstyle{definition}

\newtheorem{remark}[thm]{Remark}

\newtheorem*{acknowledgement}{Acknowledgements}

\usepackage{tikz}
\usepackage{graphicx}
\DeclareGraphicsExtensions{.pdf,.png,.jpg}
\usepackage{subfig}

\counterwithin*{equation}{section}

\newcommand{\bN}{\mathbb{N}}
\newcommand{\bC}{\mathbb{C}}

\newcommand{\bR}{\mathbb{R}}

\newcommand{\bZ}{\mathbb{Z}}

\newcommand{\bP}{\mathbb{P}}

\newcommand{\cP}{\mathcal{P}}

\newcommand{\cI}{\mathcal{I}}

\newcommand{\norm}[1]{\left\| #1 \right\|}

\DeclareMathOperator*{\supp}{supp}

\usepackage{hyperref}
\begin{document}
\title{Near-optimal restriction estimates for Cantor sets on the parabola}
	\author{Donggeun Ryou}
	\address{Department of Mathematics, University of Rochester, Rochester, NY, USA}
	\email{dryou@ur.rochester.edu}
	\date{\today}
	\subjclass{42B10 (primary) 28A80 (secondary)}
	\keywords{Random Cantor sets, Restriction estimate, Salem set}
	%\tableofcontents
	%\newpage

\begin{abstract} 
	For any $0 < \alpha <1$, we construct Cantor sets on the parabola of Hausdorff dimension $\alpha$ such that they are Salem sets and each associated measure $\nu$ satisfies the estimate $\norm{\widehat{f d\nu}}_{L^p(\bR^2)} \leq C_p \norm{f}_{L^2(\nu)}$ for all $p >6/\alpha$ and for some constant $C_p >0$ which may depend on $p$ and $\nu$. The range $p>6/\alpha$ is optimal except for the endpoint. The proof is based on the work of {\L}aba and Wang on restriction estimates for random Cantor sets and the work of Shmerkin and Suomala on Fourier decay of measures on random Cantor sets. They considered fractal subsets of $\bR^d$, while we consider fractal subsets of the parabola.
\end{abstract}
\maketitle
\section{Introduction}
Given $f : \bR^d \rightarrow \bC$, the Fourier transform of $f$ is defined by
\begin{equation*}
    \widehat{f}(\xi) = \int_{\bR^d} f(x) e^{-2\pi i x \cdot \xi} dx \qquad \forall \xi \in \bR^d.
\end{equation*} 
Let $\mu$ be a compactly supported measure on $\bR^d$. If $f$ is $\mu$-integrable, $\widehat{fd\mu}$ is defined by
\begin{equation*}
    \widehat{fd\mu}(\xi) = \int f(x) e^{-2\pi i x \cdot \xi} d\mu(x) \qquad \forall \xi \in \bR^d.
\end{equation*}
We consider a Borel probability measure $\mu$ in $\bR^d$ and assume that it satisfies the estimate
\begin{equation}\label{introest}
    \norm{\widehat{fd\mu}}_{L^p(\bR^d) } \leq C_{p} \norm{f}_{L^2(\mu)} \qquad \forall f \in L^2(\mu)
\end{equation}
for some constant $C_{p}$ which depends on $p$. It is easy to show that \eqref{introest} holds when $p=\infty$. However, we can say more if we focus on more specific measures. One typical example is a surface carried measure on the sphere or on the paraboloid. Then, \eqref{introest} holds when $p\geq 2(d+1)/(d-1)$ and the range of $p$ is sharp. This is the result of the Stein-Tomas theorem (see for example \cite{StSh11}).

Let $B(x,r)$ be a closed ball of radius $r$ centered at $x$. If we have assumptions on the Fourier decay of ${\mu}$ and the upper bound of $\mu(B(x,r))$, the following result generalizes the Stein-Tomas theorem. 
\begin{thm}\label{Mock}
    Let $\mu$ be a Borel probability measure on $\bR^d$. Assume that there exist $\alpha, \beta \in (0, d)$ and $C_1, C_2 \geq 0$ such that \begin{equation}\label{Mockcond1}
        \mu(B(x,r)) \leq C_1 r^\alpha \qquad \forall x \in \bR^d, r>0,
     \end{equation} 
     \begin{equation}\label{Mockcond2}
         |\widehat{\mu}(\xi) | \leq C_2 (1+|\xi|)^{-\beta/2} \qquad \forall \xi \in \bR^d.
    \end{equation}
     Then, the estimate \eqref{introest} holds when $p \geq (4d-4\alpha+2\beta)/\beta$ for some constant $C_p \geq 0$ which depends on $p,C_1$ and $C_2$.
\end{thm}
Mockenhaupt \cite{Mo00} and Mitsis \cite{Mit02} proved that \eqref{introest} holds when  $p > (4d-4\alpha+2\beta)/\beta$ and the endpoint result was proved by Bak and Seeger \cite{BS11}. If $\mu$ is a surface carried measure on the paraboloid, $\alpha=\beta=d-1$. Thus, Theorem \ref{Mock} recovers the range $p\geq 2(d+1)/(d-1)$.

Define the \textit{critical exponent $p_c$} of the measure $\mu$ by 
\begin{equation*}
    p_c=\inf\{p: \eqref{introest}\ \mathrm{holds }   \}.
\end{equation*}
Theorem \ref{Mock} means $p_c \leq (4d-4\alpha+2\beta)/\beta$. This upper bound of $p_c$ is known to be optimal when $d-1 < \beta \leq \alpha < d$ in the sense that there exists a measure $\mu$ such that it satisfies the assumptions of Theorem \ref{Mock} and $p_c = (4d-4\alpha+2\beta)/\beta$. When $d=1$, Hambrook and {\L}aba \cite{HL13} constructed a measure such that \eqref{Mockcond2} holds for any $\beta$ arbitrarily close to but smaller than $\alpha$ and $p_c = (4-2\alpha)/\alpha$. Chen \cite{Xchen16} extended this result to general $\alpha$ and $\beta$ such that $0 \leq \beta \leq \alpha \leq 1$. In higher dimensions, it was done in \cite{HL16} by Hambrook and {\L}aba.

However, there exist some measures $\mu$ such that $p_c < (4d-4\alpha+2\beta)/\beta$. In \cite{HL13}, Hambrook and {\L}aba showed that $p_c \geq 2d/\alpha_0$ if $\mu$ is supported on a set of Hausdorff dimension $\alpha_0$. Note that $0\leq \alpha, \beta \leq \alpha_0$. If \eqref{Mockcond1} and \eqref{Mockcond2} hold for values arbitrarily close to $\alpha_0$, the support of $\mu$ is called a \textit{Salem set}. If $0 < \alpha_0 < d$, we have
\begin{equation*}
    \frac{2d}{\alpha_0} < \frac{4d-2\alpha_0}{\alpha_0} \leq \frac{4d-4\alpha+2\beta}{\beta}.
\end{equation*}
Thus, $2d/\alpha_0$ is smaller than $(4d-4\alpha+2\beta)/\beta$ even when $\mu$ is supported on a Salem set. Examples such that $p_c = 2d/\alpha_0$ were provided by Chen \cite{Xchen14}, Chen and Seeger \cite{CS17}, Shmerkin and Suomala \cite{ShSu18} (see also \cite{ShSu17})  and {\L}aba and Wang \cite{LW18}. Therefore, the lower bound of $p_c$ is optimal for all $0 <\alpha_0 < d$ in the sense that there exists a measure $\mu$ whose $p_c = 2d/\alpha_0$ for each $\alpha_0$.

In \cite{Xchen14} and \cite{CS17}, Chen and Seeger constructed measures supported on a Salem set of Hausdorff dimension $\alpha_0$ such that $p_c = 2d/\alpha_0$ where $\alpha_0 = d/k$ and $k$ is an integer. They considered $k$-fold self-convolution of $\mu$. Shmerkin and Suomala \cite{ShSu18} constructed the example through random fractal measures. Their result covers when $d=1$ and $1/2<\alpha_0<1 $. {\L}aba and Wang \cite{LW18} constructed measures of Hausdorff dimension $\alpha_0$ such that \eqref{Mockcond2} holds for $\beta < \min ({\alpha_0/2 ,1})$ and $p_c=2d/\alpha_0$. They used $\Lambda(p)$-set and decoupling. Their construction works for all $\alpha_0$ such that $0< \alpha_0 < d$. However, in their result, the estimate \eqref{introest} only holds when $p > 2d/\alpha_0$, while the results of Chen, Chen-Seeger, and Shmerkin-Suomala work even at the endpoint. Thus, it is still open whether, for any $0 <\alpha_0 < d$, there exists a measure $\mu$ such that its support has Hausdorff dimension $\alpha_0$ and \eqref{introest} holds for $p = 2d/\alpha_0$.

We also mention that all the results are related to a question raised by Mitsis \cite{Mi02} (see also Mattila \cite{Ma04}). If $\alpha$ is not an integer such that $0<\alpha<d$, Mitsis asked if there is a compactly supported measure $\mu$ such that
\begin{equation*}
    C_1 r^\alpha \leq \mu(B(x,r)) \leq C_2 r^\alpha \qquad \forall x \in \supp(\mu), \forall r >0
\end{equation*}
for some constants $C_1, C_2>0$ and 
\begin{equation*}
    |\widehat{\mu}(\xi)| \leq |\xi|^{-\alpha/2} \qquad \forall \xi \in \bR^d.
\end{equation*}
A related result can also be found in \cite{Kor11}. However, there is no known result yet, including this paper, since we have $|\widehat{\mu}(\xi)| \leq C_{\alpha,\epsilon} |\xi|^{-\alpha/2+\epsilon}$ at best where $C_{\alpha,\epsilon}$ is a constant depends on $\alpha$ and $\epsilon$.

Throughout the paper, we denote by $X \lesssim Y$ when $X \leq CY$ for some constant $C>0$ and we write $X \approx Y$ to denote that $X \lesssim Y$ and $Y \lesssim X$. If the constant $C$ depends on a parameter such as $\epsilon$, we write $X(\epsilon) \lesssim_\epsilon Y(\epsilon)$ instead of $X(\epsilon) \leq C(\epsilon) Y(\epsilon)$ where the constant $C(\epsilon)$ depends on $\epsilon$. Depending on the context, $|\cdot |$ denotes the $\ell^2$ norm of a vector in $\bR^d$, the cardinality of a finite set or the $d$-dimensional Lebesgue measure of a subset in $\bR^d$. We abbreviate $e^{2\pi i x}$ to $e(x)$. 

Now, let us turn to our setting. Earlier works constructed fractal measures on $\bR^d$, but we consider fractal measures on the parabola $\bP^{1}:=\{(x,x^2 ) : 0 \leq x \leq 1 \}$.
\begin{thm}\label{M1}
    Let $0<\alpha<1 $. There exists a Borel probability measure $\nu$ supported on a subset of $\bP^1$ which satisfies the following. 
    \begin{itemize}
        \item[(1)] The support of the measure $\nu$ has Hausdorff dimension $\alpha$.
        \item[(2)] For any $0 < \epsilon<\alpha $, we have 
    \begin{equation}\label{M1eq1}
        \nu (B(x,r)) \lesssim_{\alpha,\epsilon} r^{\alpha-\epsilon} \qquad \forall x \in \bR^2, \forall r>0.
    \end{equation}
        \item[(3)] For any $\epsilon >0$, 
    \begin{equation}\label{M1eq2}
        |\widehat{\nu}(\xi)| \lesssim_{\alpha,\epsilon} (1+|\xi|)^{-\alpha/2+\epsilon} \qquad \forall \xi \in \bR^2. 
    \end{equation}
        \item[(4)]For every $p > 6/\alpha$, we have the estimate
    \begin{equation}\label{M1eq3}
        \norm{\widehat{fd \nu}}_{L^p(\bR^2)} \lesssim_p \norm{f}_{L^2(\nu)} \qquad \forall f \in L^2(\nu).
    \end{equation}
    Equivalently, for any $1\leq q < 6/(6-\alpha)$ we have
    \begin{equation*}
        \norm{\widehat{f}}_{L^2(\nu)} \lesssim_{q} \norm{f}_{L^q(\bR^2)} \qquad \forall f \in L^q(\bR^2).
    \end{equation*}
    \end{itemize}
\end{thm}
Note that the support of $\nu$ is a Salem set and it is a subset of the parabola, which is also a Salem set. Theorem \ref{M1} is sharp except for the endpoint in the following sense. 
\begin{thm} \label{M2}
    Let $\bP^{d-1}:= \{(x, |x|^2 : x \in [0,1]^{d-1} \}$ and $0 < \alpha < d-1$. Assume that $\nu$ is a Borel probability measure supported on a subset of $\bP^{d-1}$ whose Hausdorff dimension is $\alpha$. Assume that
    \begin{equation}\label{M2eq1}
        \nu(B(x,r)) \lesssim_{\alpha} r^{\alpha} \qquad \forall x \in \bP^{d-1}, \forall r >0.
    \end{equation}
    Then, the estimate 
    \begin{equation}\label{M2eq2}
        \norm{\widehat{fd\nu} }_{L^p(\bR^2)} \lesssim_{p,q} \norm{f}_{L^q(\nu)} \qquad \forall f \in L^q(\nu)
    \end{equation}
    cannot hold if $p < {q'(d+1)}/{\alpha}$ where $1/q+1/q'=1$. 
\end{thm}
When $d=2$ and $q=2$, Theorem \ref{M2} implies that $p_c \geq 6/\alpha$. Since $\nu$ is on the parabola, the lower bound of $p_c$ increased from $4/\alpha$ to $6/\alpha$ compared with the result of Hambrook and {\L}aba \cite{HL13}. Also, since $0 < \alpha <1$, we have $6/\alpha < (8-2\alpha)/\alpha$. Thus, it is smaller than the upper bound of $p_c$ obtained from Theorem \ref{Mock}.
\subsection{Outline of the paper} The proof was inspired by the work of {\L}aba and Wang \cite{LW18}, but we cannot apply it directly. 
The main obstacle is that we cannot make use of dual cubes as in \cite{LW18}. For example, they considered a Cantor set which can be covered by $R^{\alpha+\epsilon}$ number of cubes $Q$ of side length $R^{-1}$ for any $R>0$. Let $\psi_Q = \psi(R(x-x_Q))$  where $\psi$ is a compactly supported Schwartz function and $x_Q$ is the center of $Q$. For any $Q$, $\widehat{\psi_Q}$ is essentially supported on a cube of side length $O(R)$ centered at the origin.

Now, let us consider a subset of the parabola as follows. First, consider a Cantor set on $[0,1]$ which can be covered by $R^{\alpha+\epsilon}$ number of intervals $I$ of length $R^{-1/2}$. Then, we can consider a parallelogram $\Omega_I$ of sides $O(R^{-1/2})$ and $O(R^{-1})$ which contains the subset of the parabola above $I$, i.e. $\{ (x,x^2): x\in I \} \subseteq \Omega_I$. Let $\psi_{\Omega_I}$ be a Schwartz function supported on $\Omega_I$. Specifically, 
\begin{equation*}
    \psi_{\Omega_I} = \psi(R^{1/2}(x_1-x_{I}), R(x_2-(2x_{I}+R^{-1/2})(x_1-x_{I})-x_{I}^2))
\end{equation*}
where $x_{I}$ is the center of $I$. Then, $\widehat{\psi_{\Omega_I}}$ is essentially supported on a parallelogram of sides $O(R^{1/2})$ and $O(R)$. However, the directions of long sides of them are all different, so that they overlap in a small cube of side $O(R^{1/2})$, while the union of these parallelograms is a cube of side $O(R)$. Thus, we need new arguments in addition to the modifications of the proof in \cite{LW18}.

In Section \ref{secCan}, we construct a Cantor set on $[0,1]$ by using $\Lambda(p)$-sets and consider the associated measure which is supported on a subset of parabola above the Cantor set. Let $\nu$ be the measure constructed in Section \ref{secCan}. Then $\nu$ is supported on a set of Hausdorff dimension $\alpha$ and satisfies \eqref{M1eq1}. In Section \ref{secDec}, we obtain the decoupling inequality for the support of $\nu$. We use $\Lambda(p)$-sets and the result of \cite{CDGJLM22}. In Section \ref{secLR}, we obtain the local restriction estimate for $\nu$, which is an analogue of Corollary 2 in \cite{LW18}. However, we use the mixed norm interpolation \cite{BP61} in addition to the argument in \cite{LW18}. In Section \ref{secFD}, we obtain the Fourier decay of $\nu$. We only need a negative exponent to derive \eqref{M1eq3}, but it turned out that almost surely $\nu$ has the Fourier decay arbitrary close to the optimal one, as in \eqref{M1eq2}. We use the arguments in \cite{Kauf76} and \cite{ShSu17} in our setting. The local restriction estimate and the Fourier decay lead to the global restriction estimate in Section \ref{secGR}, which is \eqref{M1eq3}. Tao's epsilon removal lemma \cite{Tao99} is used as in \cite{LW18}, but we simplify the proof of Lemma 9 in \cite{LW18} in order to avoid the problem we described above. Then, we prove that Theorem \ref{M1} happens almost surely if $\nu$ is a random measure constructed in Section \ref{secGR}. Lastly, we prove Theorem \ref{M2} in Section \ref{secSharp}.

\begin{acknowledgement}
The author would like to thank his advisor Alex Iosevich for many discussions of this work and encouragement. The author would also like to thank Shaoming Guo, Zane Kun Li for helpful conversations and thank Kyle Hambrook for the reference \cite{Kauf76}. The author is grateful to the anonymous referees for many valuable comments and suggestions.
\end{acknowledgement}

\section{The construction of a Cantor set}\label{secCan}
Let $\{n_j\}_{j \in \bN \cup \{ 0 \} }$ be a sequence of positive integers, and let $N_j = n_0\cdots n_j$. We assume the following conditions on $n_j$.
\begin{equation}\label{nvcond1}
    n_0=1 \qquad \mathrm{and} \qquad 2\leq n_1 \leq n_2 \leq \cdots \leq n_j \leq \cdots ,
\end{equation}
\begin{equation}\label{nvcond2}
    \forall a >0, a^j \lesssim_a N_j,
\end{equation}
\begin{equation}\label{nvcond3}
    \forall \epsilon >0 \ \mathrm{and} \ \forall j \in \bN, n_{j+1} \lesssim_{\epsilon} N_j^\epsilon,
\end{equation}
\begin{equation}\label{nvcond4}
    \exists B \ \mathrm{such \ that } \ \forall j \in \bN, N_{2j}^{1/2} N_j^{-1} \leq B^j.
\end{equation}
For example, if $n_j \approx j $, all conditions from \eqref{nvcond1} to \eqref{nvcond4} are satisfied by Stirling's formula (see \cite[p. 98]{Dur19}). The condition \eqref{nvcond4} was not assumed in \cite{LW18}, but it is necessary in Sections \ref{secDec} and \ref{secLR}. We need the following theorem in order to use $\Lambda(p)$-sets.
\begin{thm}[Existence of $\Lambda(p)$-sets] \label{Lambdap} Let $p>2$. For every $N \in \bN$, there exists a set $S \subset \{0,1,2, \cdots, N-1\}$ such that $N^{2/p} \leq |S| < N^{2/p}+1 $ and
\begin{equation}\label{Lambdapeq}
    \norm{\sum_{a\in S}c_a e(ax)}_{L^p([0,1])} \leq C_p  \left( \sum_{a\in S} |c_a|^2\right)^{1/2} \qquad \forall \{ c_a\}_{a \in S} \in \ell^2
\end{equation}
for some constant $C_p > 0$ which depends only on $p$, but not on $N$.
\end{thm}
Theorem \ref{Lambdap} was first proved by Bourgain \cite{Bour89} and another proof was given by Talagrand \cite{Tal95}, see also \cite[Section 19.3]{Tal21} for simpler proof.

Let $0 < \alpha <1$ and for each $j \in \bN$, let $\{t_j\}_{\bN}$ be a sequence of integers such that
\begin{equation*}
    n_j^{\alpha/2} \leq t_j < n_j^{\alpha/2}+1.  
\end{equation*}
Define
\begin{equation*}
    \Sigma_j := \{ S \subset [0, n_j^{1/2}) \cap \bZ : |S|=t_j \ \mathrm{and}\  \eqref{Lambdapeq} \ \mathrm{holds \ for \ } p=2/\alpha,  N = n_j^{1/2} \}.
\end{equation*}
The set $\Sigma_j$ is non-empty because of Theorem \ref{Lambdap}.

Let $E_0 = [0,1]$ and $A_0 = \{0\}$. For each $ a \in A_j$, choose a set $S_{j+1,a} \in \Sigma_{j+1}$ and let
\begin{equation*}
    A_{j+1,a}= n_{j+1}^{1/2}a+ S_{j+1,a} \qquad A_{j+1} = \cup_{a \in A_j} A_{j+1,a} \qquad E_{j+1} = N_{j+1}^{-1/2} ( A_{j+1} + [0, 1]).
\end{equation*}
The set $A_{j}$ is a subset of $\{ 0,1,2, \cdots, N_j^{1/2}-1 \}$ and $ [0,1] \supseteq E_1 \supseteq \cdots \supseteq E_j \supseteq E_{j+1} $. Let $E_\infty :=  \cap_{j=1}^\infty E_j $. Similarly, let $P_j := \{(x_1,x_1^2) : x_1\in E_j \}$ and $P_\infty :=\cap_{j=1}^\infty P_j $.

For each $j \in \bN \cup \{0\}$, we define
\begin{equation*}
    \mu_j := \frac{1}{|E_j|} \mathbf{1}_{E_j}(x_1), \qquad x_1 \in \bR.
\end{equation*}
We identify the function $\mu_j$ with the absolutely continuous measure $\mu_jdx_1$ so that 
\begin{equation*}
    \int g d\mu_j := \int g \mu_j dx_1.
\end{equation*}
From $\mu_j$, we define the measure $\nu_j$ as follows.
\begin{equation*}
    \int f d\nu_j := \int f(x_1,x_1^2) d\mu_j.
\end{equation*}
For simplicity, define $\norm{\mu_j} := \mu_j(\bR)$ and $\norm{\nu_j} := \nu_j(\bR^2)$. The measures $\mu_j$ and $\nu_j$ converge weakly as $j \rightarrow \infty$ to probability measures $\mu$ and $\nu$ supported on $E_\infty$ and $P_\infty$ respectively.

\begin{lem}[Lemma 6, \cite{LW18}]\label{CanDim} The measure $\mu $ and the set $E_\infty$ constructed above satisfy the following.\\  
(1) The set $E_\infty$ has Hausdorff dimension $\alpha$.\\
(2) For any $0< \epsilon <\alpha $, we have
\begin{equation*}
    \mu(B(x,r)) \lesssim_{\alpha,\epsilon} r^{\alpha-\epsilon} \qquad \forall \xi \in \bR, \forall r>0.
\end{equation*}
\end{lem}
It is easy to show that the same is true for $\nu$. 
\begin{lem}\label{CanDimPara} The measure $\nu$ and the set $P_\infty$ constructed above satisfy the following. \\
(1) The set $P_\infty$ has Hausdorff dimension $\alpha$.\\
(2) For any $0< \epsilon <\alpha $, we have
\begin{equation*}
    \nu(B(x,r)) \lesssim_{\alpha,\epsilon} r^{\alpha-\epsilon} \qquad \forall \xi \in \bR^2, \forall r>0.
\end{equation*}
\end{lem}
\section{Decoupling inequalities}\label{secDec}
In this section, we will derive the decoupling estimates for $E_\infty$ and $P_\infty$. We will define decoupling constants and discrete restriction constants, and we will verify relations between them. Since we need to look into $E_j$ in different scales, we define the following.

For $0 \leq i \leq j$ and $a \in A_{i}$, we define $E_{j,i,a} =E_j \cap N_i^{-1/2}(a+ [0,1]) $, see Figure \ref{fig1}. If $i=0$, since $A_0=\{0\}$, note that $E_{j,i,a} = E_{j,0,0}=E_j$. Let $\cP_j(E_{j,i,a}) $ be the partition of $E_{j,i,a}$ into intervals of length $N_j^{-1/2}$ and $\widehat{f_I}(\xi_1) := \widehat{f_I}(\xi_1) \mathbf{1}_I(\xi_1)$ for $I \in \cP_{j}(E_{j,i,a})$ and $\xi_1 \in \bR$. Sometimes, we will write $I(b) \in \cP_{j}(E_{j,i,a})$ for $N_j^{-1/2} (b +[0,1]) \subseteq E_{j,i,a}$ for some $b \in A_j$. 

\begin{figure}
    \centering
\tikzset{every picture/.style={line width=0.75pt}} 

\begin{tikzpicture}[x=0.75pt,y=0.75pt,yscale=-1,xscale=1]
%uncomment if require: \path (0,174); %set diagram left start at 0, and has height of 174

\draw    (148,47.69) -- (496.4,47.69) ;
\draw  [color={rgb, 255:red, 0; green, 0; blue, 0 }  ,draw opacity=0 ][fill={rgb, 255:red, 0; green, 0; blue, 0 }  ,fill opacity=1 ] (200.28,44.23) -- (230.06,44.23) -- (230.06,47.3) -- (200.28,47.3) -- cycle ;
\draw    (200.28,37.69) -- (200.28,47.8) ;
\draw    (203.88,37.69) -- (228.28,37.69) ;
\draw [shift={(231.28,37.69)}, rotate = 180] [fill={rgb, 255:red, 0; green, 0; blue, 0 }  ][line width=0.08]  [draw opacity=0] (5.36,-2.57) -- (0,0) -- (5.36,2.57) -- cycle    ;
\draw [shift={(200.88,37.69)}, rotate = 0] [fill={rgb, 255:red, 0; green, 0; blue, 0 }  ][line width=0.08]  [draw opacity=0] (5.36,-2.57) -- (0,0) -- (5.36,2.57) -- cycle    ;
\draw  [color={rgb, 255:red, 0; green, 0; blue, 0 }  ,draw opacity=0 ][fill={rgb, 255:red, 0; green, 0; blue, 0 }  ,fill opacity=1 ] (290.55,44.23) -- (320.33,44.23) -- (320.33,47.3) -- (290.55,47.3) -- cycle ;
\draw  [color={rgb, 255:red, 0; green, 0; blue, 0 }  ,draw opacity=0 ][fill={rgb, 255:red, 0; green, 0; blue, 0 }  ,fill opacity=1 ] (260.78,44.23) -- (290.55,44.23) -- (290.55,47.3) -- (260.78,47.3) -- cycle ;
\draw  [color={rgb, 255:red, 0; green, 0; blue, 0 }  ,draw opacity=0 ][fill={rgb, 255:red, 0; green, 0; blue, 0 }  ,fill opacity=1 ] (411,44.23) -- (440.78,44.23) -- (440.78,47.3) -- (411,47.3) -- cycle ;
\draw  [color={rgb, 255:red, 0; green, 0; blue, 0 }  ,draw opacity=0 ][fill={rgb, 255:red, 0; green, 0; blue, 0 }  ,fill opacity=1 ] (351.53,44.23) -- (381.3,44.23) -- (381.3,47.3) -- (351.53,47.3) -- cycle ;
\draw    (230.06,37.19) -- (230.06,47.3) ;
\draw    (200.28,47.3) -- (200.28,57.42) ;
\draw    (440.78,47.8) -- (440.78,57.92) ;
\draw    (203.28,57.42) -- (437.78,57.42) ;
\draw [shift={(440.78,57.42)}, rotate = 180] [fill={rgb, 255:red, 0; green, 0; blue, 0 }  ][line width=0.08]  [draw opacity=0] (5.36,-2.57) -- (0,0) -- (5.36,2.57) -- cycle    ;
\draw [shift={(200.28,57.42)}, rotate = 0] [fill={rgb, 255:red, 0; green, 0; blue, 0 }  ][line width=0.08]  [draw opacity=0] (5.36,-2.57) -- (0,0) -- (5.36,2.57) -- cycle    ;

% Text Node
\draw (178.9,66.55) node [anchor=north west][inner sep=0.75pt]  [font=\small]  {$N_{i}^{-1/2} a$};
\draw (201.27,9.65) node [anchor=north west][inner sep=0.75pt]  [font=\small]  {$N_{j}^{-1/2}$};
\draw (302.27,66.4) node [anchor=north west][inner sep=0.75pt]  [font=\small]  {$N_{i}^{-1/2}$};

\end{tikzpicture}
    \caption{$E_{j,i,a}$}
    \label{fig1}
\end{figure}

For $p \geq 2$, let $D_p(E_{j,i,a})$ denote the best constant such that
\begin{equation}\label{defD_pE_j,i}
    \norm{\sum_{I \in \cP_{j}(E_{j,i,a}) } f_I }_{L^p(\bR)} \leq D_p(E_{j,i,a}) \left( \sum_{ I \in \cP_{j}(E_{j,i,a})} \norm{f_I}_{L^p(\bR)}^2 \right)^{1/2}
\end{equation}
for any function $f$ such that $\supp \widehat{f} \subseteq E_{j,i,a}$ and let $D_{p}(E_{j,i}) =\sup_{a \in A_i} D_{p}(E_{j,i,a})$. 

If $I \in \cP_{j}(E_{j,i,a})$, there exist an unique element $b \in {A_j}$ such that $I(b) \in \cP_j(E_{j,i,a})$. Thus, we define  
\begin{equation*}
	\begin{split}
		\Omega_I = \{\xi \in \bR^2 : &bN_{j}^{-1/2} \leq \xi_1 \leq (b+1)N_{j}^{-1/2},\\
		&|\xi_2 - (2b+1)N_{j}^{-1/2}(\xi_1 - bN_{j}^{-1/2}) - b^2N_{j}^{-1}| \leq N_{j}^{-1}\}
	\end{split}
\end{equation*}
and $P_{j,i,a}  = \cup_{I \in \cP_j(E_{j,i,a})} \Omega_I$ and $f_{\Omega_I} = \widehat{f}(\xi) \mathbf{1}_{\Omega_I}(\xi)$ for $\xi \in \bR^2$. Let $D_p(P_{j,i,a})$ be the best constant such that 
\begin{equation}\label{defD_pP_j,i}
	\norm{\sum_{I \in \cP_{j}(E_{j,i,a}) } f_{\Omega_I} }_{L^p(\bR^2 )} \leq D_p(P_{j,i,a}) \left( \sum_{ I \in \cP_{j}(E_{j,i,a})} \norm{f_{\Omega_I}}_{L^p(\bR^2)}^2 \right)^{1/2}
\end{equation}
for any function $f$ such that $\supp \widehat{f} \subseteq P_{j,i,a}$. Similarly, let $D_p(P_{j,i}) = \sup_{a \in A_i} D_p(P_{j,i,a})$.

We also need discrete versions of $D_p(E_{j,i})$ and $D_p(P_{j,i})$. For $a \in A_i$, let us define
\begin{equation*}
	A_{j,i,a} = A_j \cap (N_jN_i^{-1})^{1/2} (a+[0,1)),
\end{equation*}
i.e., $A_{j,i,a}$ is the set of $j$-th level descendants of $a \in A_i$. Then, let $K_p(E_{j,i,a})$ be the best constant such that
\begin{equation*}
    \norm{\sum_{b \in A_{j,i,a}} c_b e(bx)}_{L^p([0,1])} \leq K_p(E_{j,i,a}) \left(\sum_{b \in A_{j,i,a}} |c_b|^2\right)^{1/2}
\end{equation*}
for any $\{c_b\}_{b \in A_{j,i,a}} \in \ell^2$ and $K_p(E_{j,i}) := \sup_{a \in A_i} K_p(E_{j,i,a})$. Similarly, let $K_p(P_{j,i,a})$ be the best constant such that 
\begin{equation*}
    \norm{\sum_{b \in A_{j,i,a}} c_b e({bx_1 + b^2x_2})}_{L^p([0,1]^2)} \leq K_p (P_{j,i,a}) \left(\sum_{b \in A_{j,i,a}}|c_b|^2 \right)^{1/2}
\end{equation*}
for any $\{c_b\}_{b \in A_{j,i,a}} \in \ell^2$ and $K_p(P_{j,i}) := \sup_{a \in A_i} K_p(P_{j,i,a})$.

In short, $D_p(E_{j,i,a})$ and $D_p(P_{j,i,a})$ are decoupling constants for $E_{j,i,a}$ and $P_{j,i,a}$ respectively and $K_p(E_{j,i,a})$ and $K_p(P_{j,i,a})$ are discrete restriction constants for $E_{j,i,a}$ and $P_{j,i,a}$ respectively. If $i=0$, we simply write $D_p(E_{j}) := D_p(E_{j,0})$, $D_p(P_j) :=D_p(P_{j,0})$, $K_p(E_j):= K_p(E_{j,0})$ and $K_p(P_{j}):=K_p(P_{j,0})$. 

We will show that $K_{6/\alpha}(P_j) \lesssim_{\epsilon} N_j^\epsilon $ through the following inequalities:
\begin{equation*}
    K_{6/\alpha}({P_j}) \lesssim_{\alpha} D_{6/\alpha}(P_j) \lesssim_\epsilon N_j^\epsilon D_{2/\alpha}(E_j)
\end{equation*}
and
\begin{equation*}
    D_{2/\alpha}(E_j) \approx_\alpha K_{2/\alpha}(E_j) \lesssim_\epsilon \widetilde{C}_{2/\alpha}^j
\end{equation*}
where $\widetilde{C}_p$ is a constant multiple of $C_p$ in \eqref{Lambdapeq}. 

\subsection{From \texorpdfstring{$\Lambda(p)$-sets}{lambda(p) sets} to decoupling for \texorpdfstring{$E_j$}{Ej} }
We need the following lemmas to use $\Lambda(p)$-sets in multiscale.
\begin{lem}\label{DIS_1S_2}
    For $p \geq 2$, let $S_1$ be a subset of $\bZ^d$. For $a \in S_1$ and $k \in \bN$, let $S_{2,a}$ be subsets of $([0,k) \cap \bZ)^d$. Assume that the sets $S_1$ and $S_{2,a}$ satisfy
    \begin{equation}\label{DIS_1S_2eq1}
        \norm{\sum_{a \in S_1} c_a e(a \cdot x)}_{L^p([0,1]^d)} \leq C_1 \left( \sum_{a \in S_1} |c_a|^2 \right)^{1/2} \qquad \forall \{ c_a \}_{a \in S_1} \in \ell^2
    \end{equation}
    and
    \begin{equation}\label{DIS_1S_2eq2}
        \norm{\sum_{b \in S_{2,a}} c_b e(b \cdot x)}_{L^p([0,1]^d)} \leq C_2 \left( \sum_{b \in S_2} |c_b|^2 \right)^{1/2} \qquad \forall \{ c_b \}_{b \in S_2} \in \ell^2, \forall a \in S_1.
    \end{equation}
    Then, we have
    \begin{equation}\label{DIS_1S_2result}
        \norm{\sum_{a \in S_1} \sum_{b \in S_{2,a}} c_{a,b} e((ka+b)\cdot x) }_{L^p([0,1]^d)} \lesssim_p C_1C_2 \left( \sum_{a \in S_1} \sum_{ b\in S_{2,a} } |c_{a,b}|^2 \right)^{1/2}
    \end{equation}
    for all $\{ c_{a,b} \}_{a\in S_1, b \in S_{2,a}} \in \ell^2$.
\end{lem}
\begin{lem}\label{DIequiv}
    For $p\geq 2$ and for any $a \in A_i$, $D_p(E_{j,i,a}) \approx_p K_p(E_{j,i,a})$.
\end{lem}

Lemma \ref{DIS_1S_2} and \ref{DIequiv} follow from the proofs of Lemma 4 and Proposition 1 in \cite{LW18} and the duality of $L^p$ played a key role. In Section \ref{secAppen}, we provide alternative proofs of them which do not rely on duality.
\begin{remark}
    Readers may want to compare Lemma \ref{DIS_1S_2} with Proposition 3.5 in \cite{CDGJLM22}. They are similar but there is a trade-off. Lemma \ref{DIS_1S_2} can cover more general cases, because $p$ does not need to be an even integer and there is no assumption on carryover. However, the implicit constant of the inequality is larger than Proposition 3.5 in \cite{CDGJLM22} because the constant in \eqref{DIS_1S_2result} is not exactly $C_1C_2$, but $C_p' C_1C_2$ where $C_p' >1$.
\end{remark}
\begin{lem}\label{DecE_j,i}
    For $0 < \alpha <1$, $D_{2/\alpha}(E_{j,i}) \lesssim_\alpha \widetilde{C}_{2/\alpha}^{j-i}$ where the $\widetilde{C}_{2/\alpha}$ is a constant multiple of $C_{2/\alpha}$ in \eqref{Lambdapeq}.
\end{lem}
\begin{proof}
    For any $a \in A_i$, note that
    \begin{equation*}
        \sum_{b \in A_{j,i,a}}c_b e(bx) = \sum_{b_1 \in A_{j-1,i,a} } \sum_{b_2 \in S_{j,b_1}} c_{b_1,b_2} e(n_j^{1/2} b_1 +b_2).
    \end{equation*}
    For fixed $a \in A_i$, let $C_1$ be a constant such that 
    \begin{equation*}
        \norm{\sum_{b_1 \in A_{j-1,i,a}} c_{b_1} e(b_1 x)}_{L^{2/\alpha}([0,1])} \leq C_1 \left( \sum_{b_1 \in A_{j-1,i,a}} |c_{b_1}|^2 \right)^{1/2} \qquad \forall \{ c_{b_1} \}_{b_1 \in A_{j-1,i,a}} \in \ell^2.
    \end{equation*}
    Combining Theorem \ref{Lambdap} and Lemma \ref{DIS_1S_2}, we get
    \begin{equation*}
        \norm{\sum_{b \in A_{j,i,a}}c_b e(bx)}_{L^{2/\alpha}([0,1])} \leq C_1 \widetilde{C}_{2/\alpha}  \left(\sum_{b \in A_{j,i,a}}|c_b|^2 \right)^{1/2} \qquad \forall \{ c_b \}_{b \in A_{j,i,a}} \in \ell^2.
    \end{equation*}
    where $\widetilde{C}_{2/\alpha}$ is a constant multiple of $C_{2/\alpha}$.
    
    Iterating Theorem \ref{Lambdap} and Lemma \ref{DIS_1S_2}, we get $K_{2/\alpha}(E_{j,i,a}) \lesssim \widetilde{C}_{2/\alpha}^{j-i}$. By Lemma \ref{DIequiv}, we obtain
    \begin{equation*}
        D_{2/\alpha}(E_{j,i,a}) \approx_\alpha K_{2/\alpha}(E_{j,i,a}) \lesssim \widetilde{C}_{2/\alpha}^{j-i}.
    \end{equation*}
    for any $a \in A_i$. Taking the supremum over $a \in A_{i}$ finishes the proof.
\end{proof}

\subsection{From decoupling for \texorpdfstring{$E_j$}{Ej} to decoupling for \texorpdfstring{$P_j$}{Pj}}
In this section, we write $N_{j,i} = N_jN_i^{-1}$. Also, note that $N_{j,0} = N_j$. We will prove the following decoupling estimate for $P_j$.
\begin{prop}\label{DecP_k}
    For $0 < \alpha <1$, we have $D_{6/\alpha} (P_{j}) \lesssim_{\alpha, \epsilon} N_{j}^\epsilon$.
\end{prop}
It is well known that $K_p(P_j) \lesssim_p D_p(P_j)$ (see for example \cite[Theorem 13.1]{Dem20}). Thus, we have $K_{6/\alpha} (P_j) \lesssim_{\alpha,\epsilon} N_j^\epsilon$. We will use it in the proof of Proposition \ref{PropLRresult}.

In \cite{CDGJLM22}, they proved that the decoupling estimate for a Cantor set on the parabola can be derived from the decoupling estimate for a Cantor set on the line. We adapt their argument to our setting. The proofs of Lemma \ref{DecAM} and \ref{Decbil} are standard, but we included them for convenience.
\begin{lem}[Almost multiplicativity]\label{DecAM} For $i,i_0,j \geq 0$ such that $i \geq i_0$, we have
\begin{equation*} 
    D_p(P_{j+i,i_0}) \leq D_p(P_{j+i,i}) D_p(P_{i,i_0}).
\end{equation*}
\end{lem}
\begin{proof}
	For arbitrary fixed $a \in A_{i_0} $ and $I_i(b) = N_i^{-1/2}(b+[0,1])$ where $ b\in \bZ$, 
	\begin{equation*}
		\sum_{I \in \cP_{j+i}(E_{j+i,i_0,a})} f_{\Omega_I} = \sum_{I_i(b) \in \cP_i(E_{i,i_0, a})} \sum_{I\in \cP_{j+i} (E_{{j+i},i, b})}f_{\Omega_I}.
	\end{equation*}
    For any function $f$ such that $\supp \widehat{f} \subseteq P_{j+i,i_0,a}$, we obtain 
	\begin{equation*}
		\begin{split}
			&\norm{\sum_{I \in \cP_{j+i}(E_{j+i,i_0,a})} f_{\Omega_I}}_{L^p(\bR^2)} \\
			&\leq D_p(P_{i,i_0,a})\left( \sum_{I_i(b) \in \cP_i(E_{i,i_0, a})} \norm{ \sum_{I\in \cP_{j+i} (E_{{j+i},i, b})}f_{\Omega_I}}_{L^p(\bR^2)}^2 \right)^{1/2}\\
			&\leq D_p(P_{i,i_0,a})\left( \sum_{I_i(b) \in \cP_i(E_{i,i_0, a})} D_{p}(P_{j+i,i,b})  \sum_{I\in \cP_{j+i} (E_{{j+i},i, b})}\norm{f_{\Omega_I}}_{L^p(\bR^2)}^2 \right)^{1/2}\\
			&\leq D_{p}(P_{j+i,i}) D_p(P_{i,i_0})\left( \sum_{I_i(b) \in \cP_i(E_{i,i_0, a})}  \sum_{I\in \cP_{j+i} (E_{{j+i},i, b})}\norm{f_{\Omega_I}}_{L^p(\bR^2)}^2 \right)^{1/2}.
		\end{split}
        \end{equation*} 
\end{proof}

For $0 \leq k \leq i_1, i_2 \leq j$ and $i_0 \geq 0$, we define the bilinear constant $M_{p,i_0}(j,k,i_1,i_2,a)$ which is the smallest constant such that
\begin{equation*}
    \begin{split}
        & \int_{\bR^2} \bigg|\sum_{J_1 \in \cP_{j+i_0}(I_1 \cap E_{j+i_0})} f_{\Omega_{J_1}} \bigg|^p  \bigg|\sum_{J_2 \in \cP_{j+i_0}(I_2 \cap E_{j+i_0})} g_{\Omega_{J_2}} \bigg|^{2p}dx \\
        &\leq M_{p,i_0}(j,k,i_1,i_2,a)^{3p}\times\\
        &\qquad \left( \sum_{J_1 \in \cP_{j+i_0}(I_1 \cap E_{j+i_0})} \norm{f_{\Omega_{J_1}}}_{L^{3p}(\bR^2)}^2 \right)^{p/2} \left( \sum_{J_2 \in \cP_{j+i_0}(I_2 \cap E_{j+i_0})} \norm{g_{\Omega_{J_2}}}_{L^{3p}(\bR^2)}^2 \right)^p
    \end{split}
\end{equation*}
for all $I_1 \in \cP_{i_1+i_0}(E_{i_1+i_0,i_0,a})$ and $I_2 \in \cP_{i_2+i_0}(E_{i_2+i_0,i_0,a}) $ such that $d(I_1, I_2) \geq N_{k+i_0}^{-1/2}$ and for any $f$ and $g$ such that $\supp \widehat{f} \subseteq  P_{j+i_0,i_0,a} $ and $\supp \widehat{g} \subseteq P_{j+i_0,i_0,a}$ respectively. Let $M_{p,i_0}(j,k,i_1,i_2) = \sup_{a \in A_{i_0}} M_{p,i_0}(j,k,i_1,i_2,a)$.

For any $a \in A_{i_0}$, note that
\begin{equation*}
\begin{split}
    M_{p,i_0} (j,k,i_1,i_2,a) &\leq D_{3p}(P_{j+i_0,i_1+i_0})^{1/3} D_{3p}(P_{j+i_0,i_2+i_0})^{2/3}\\
    &\leq D_{3p}(P_{j+i_0, \min(i_1,i_2)+i_0}).
\end{split}
\end{equation*}
Thus, we get
\begin{equation}\label{multi-lin_dec}
    M_{p,i_0}(j,k,i_1,i_2) \leq D_{3p}(P_{j+i_0,\min(i_1,i_2)+i_0}).
\end{equation}
We will use it later in the proof.

\begin{lem}[Bilinear reduction]\label{Decbil} If $i \leq j$ and $i_0\geq 0$, then 
\begin{equation*}
    D_{3p}(P_{j+i_0,i_0}) \lesssim D_{3p} (P_{j+i_0,i+i_0}) +N_{i+i_0,i_0}^{2} M_{p,i_0}(j,i,i,i).
\end{equation*}
\end{lem}
\begin{proof}
    For any $a \in A_{i_0}$ and for any $f$ with $\supp \widehat{f} \subseteq P_{j+i_0,i_0,a} $, we have
    \begin{align}
        &\norm{\sum_{J \in \cP_{j+i_0} (E_{j+i_0,i_0,a})} f_{\Omega_J}}_{L^{3p}(\bR^2)}^2 \nonumber \\
            &= \norm{ \sum_{I_1, I_2 \in \cP_{i+i_0} (E_{i+i_0,i_0,a})} \left( \sum_{J_1 \in \cP_{j+i_0} (I_1 \cap E_{j+i_0})} f_{\Omega_{J_1}} \sum_{J_2 \in \cP_{j+i_0} (I_2 \cap E_{j+i_0})} \overline{f_{\Omega_{J_2}}} \right) }_{L^{3p/2}(\bR^2)} \nonumber \\
            &\leq \norm{\sum_{\substack{I_1,I_2 \in \cP_{i+i_0}(E_{i+i_0,i_0,a}) \\ d(I_1,I_2 ) \leq N_{i+i_0}^{-1/2} } } (\cdots ) }_{L^{3p/2}(\bR^2)} + \norm{\sum_{\substack{I_1,I_2 \in \cP_{i+i_0}(E_{i+i_0,i_0,a}) \\ d(I_1,I_2 ) \geq N_{i+i_0}^{-1/2} } } (\cdots ) }_{L^{3p/2}(\bR^2)}. \label{Decbileq1}
    \end{align}
    By the Cauchy-Schwarz inequality, the first term of \eqref{Decbileq1} is bounded by
    \begin{equation*}
        \begin{split}
            &\leq \sum_{\substack{I_1,I_2 \in \cP_{i+i_0}(E_{i+i_0,i_0,a}) \\ d(I_1,I_2 ) \leq N_{i+i_0}^{-1/2} } } \norm{ \sum_{J_1 \in \cP_{j+i_0} (I_1 \cap E_{j+i_0})} f_{\Omega_{J_1}} }_{L^{3p}(\bR^2)}\norm{ \sum_{J_2 \in \cP_{j+i_0} (I_2 \cap E_{j+i_0})} f_{\Omega_{J_2}} }_{L^{3p}(\bR^2)}\\
            &\lesssim \sum_{I \in \cP_{i+i_0}(E_{i+i_0,i_0,a})} \norm{\sum_{J \in \cP_{j+i_0} (I \cap E_{j+i_0}) } f_{\Omega_J} }_{L^{3p}(\bR^2)}^2 \\
            &\leq \sum_{I(b)\in \cP_{i+i_0} (E_{i+i_0, i_0,a})} D_{3p}(P_{j+i_0,i+i_0,b})^2 \sum_{J \in \cP_{j+i_0}(I(b) \cap E_{j+i_0}) } \norm{f_{\Omega_J}}_{L^{3p}(\bR^2)}^2\\
            &\leq D_{3p}(P_{j+i_0, i+i_0})^2 \sum_{J \in \cP_{j+i_0}(E_{j+i_0,i_0,a}) } \norm{f_{\Omega_J}}_{L^{3p}(\bR^2)}^2.
        \end{split}
    \end{equation*}
   By the Cauchy-Schwarz inequality, $\int F^{3p/2} G^{3p/2} dx \leq (\int F^p G^{2p})^{1/2} (\int F^{2p} G^{p})^{1/2}$ for any two nonegative functions. Thus, the second term of \eqref{Decbileq1} is bounded by
    \begin{equation*}
    \begin{split}
            &\leq N_{i+i_0,i_0}^{2} \max_{\substack{ I_1,I_2 \in  \cP_{i+i_0} (E_{i+i_0},i_0,a) \\ d(I_1,I_2) \geq N_i^{-1/2}}} \norm{  \left( \sum_{J_1 \in \cP_{j+i_0} (I_1 \cap E_{j+i_0})} f_{\Omega_{J_1}}  \sum_{J_2 \in \cP_{j+i_0} (I_2 \cap E_{j+i_0})} f_{\Omega_{J_2}} \right) }_{L^{3p/2}(\bR^2)}\\
            &\leq N_{i+i_0,i_0}^{2}M_{p,i_0}(j,i,i,i)^2 \max_{\substack{ I_1,I_2 \in  \cP_{i+i_0} (E_{i+i_0},i_0,a) \\ d(I_1,I_2) \geq N_i^{-1/2}}} \begin{aligned}[t]
                \left( \sum_{J_1 \in \cP_{j+i_0} (I_1 \cap E_{j+i_0})} \norm{f_{\Omega_{J_1}}}_{L^{3p}(\bR^2)}^2 \right)^{1/2}\\ \left( \sum_{J_2 \in \cP_{j+i_0} (I_2 \cap E_{j+i_0})} \norm{f_{\Omega_{J_2}}}_{L^{3p}(\bR^2)}^2 \right)^{1/2}
            \end{aligned}\\
            &\leq N_{i+i_0,i_0}^{2} M_{p,i_0}(j,i,i,i)^2 \sum_{J \in \cP_{j+i_0} (E_{j+i_0,i_0,a})} \norm{f_{\Omega_J}}_{L^{3p}(\bR^2)}^2.
    \end{split}
    \end{equation*}
\end{proof}
\begin{lem}[Key estimate in \cite{CDGJLM22}]\label{Deckeylem}
    Let $p = 2/\alpha $ where $ 0 < \alpha <1$. If $0 \leq k \leq i$, $2i \leq j$ and $i_0 \geq 0$, then
    \begin{equation}\label{keyest1}
        M_{p,i_0}(j,k,i, i) \lesssim_{p} N_{k+i_0,i_0}^{1/6} ( \widetilde{C}_p B)^{i/3} B^{i_0/3} M_{p,i_0}(j,k,2i, i).
    \end{equation}
    If $0 \leq k \leq i$, $ 4i \leq j$ and $i_0 \geq 0$, then
    \begin{equation}\label{keyest}
        M_{p,i_0}(j,k,i, 2i) \lesssim_{p} N_{k+i_0,i_0}^{1/3} (\widetilde{C}_p B)^i B^{2i_0/3} M_{p,i_0}(j,k,4i, 2i).
    \end{equation}
    where $B$ is the constant in \eqref{nvcond4} and $\widetilde{C}_p$ is the constant in Lemma \ref{DecE_j,i}.
\end{lem}
\begin{proof}
    We follow the proof of Lemma 2.4 in \cite{CDGJLM22} with modifications. The condition \eqref{nvcond4} comes into play since $n_k$ is nondecreasing. Let us prove \eqref{keyest1} first. Fix an arbitrary $a \in A_{i_0}$. We only need to consider when $f$ and $g$ satisfy
    \begin{equation}\label{f,gnormal}
        \sum_{J_1\in \cP_{j+i_0( I_1 \cap E_{j+i_0}) }} \norm{f_{\Omega_{J_1}}}_{L^{3p}(\bR^2)}^2 =\sum_{J_2\in \cP_{j+i_0( I_2 \cap E_{j+i_0}) }} \norm{g_{\Omega_{J_2}}}_{L^{3p}(\bR^2)}^2 = 1
    \end{equation}
    where $I_1$ and $I_2$ are intervals in the definition of $M_{p,i_0}(j,k,i_1,i_2,a)$. Then, it suffices to show that,
    \begin{equation}\label{keyest1nrm}
    \begin{split}
        \int_{\bR^2} \bigg|\sum_{J_1 \in \cP_{j+i_0}(I_1 \cap E_{j+i_0})} f_{\Omega_{J_1}} \bigg|^p & \bigg|\sum_{J_2 \in \cP_{j+i_0}(I_2 \cap E_{j+i_0})} g_{\Omega_{J_2}} \bigg|^{2p} dx\\
        &\lesssim_p N_{k+i_0,i_0}^{p/2} ( \widetilde{C}_p B)^{ip} B^{i_0p} M_{p,i_0}(j,k,2i, i)^{3p}.
    \end{split}
    \end{equation}
    Let $I_1= N_{i+i_0}^{-1/2}(b_1+[0,1])$ and $I_2= N_{i+i_0}^{-1/2}(b_2+[0,1])$ for some $b_1, b_2 \in A_{i+i_0}$. We can also assume that $I_2$ is to the left of $I_1$ so that $b_1-b_2 \geq N_{i+i_0}^{1/2} N_{k+i_0}^{-1/2}$. Using the change of variables
    \begin{equation*}
        \begin{split}
            x_1&=x_1' +b_2 N_{i+i_0}^{-1/2}\\
            x_2&=x_2'+2b_2N_{i+i_0}^{-1/2}x_1'+b_2^2N_{i+i_0}^{-1},
        \end{split}
    \end{equation*}
    we can reduce \eqref{keyest1nrm} to when $I_2=[0,N_{i+i_0}^{-1/2}]$ and $I_1= N_{i+i_0}^{-1/2}((b_1-b_2+[0,1])$. For each $J\in \cP_{2i + i_0}(I_1 \cap E_{2i+i_0})$, the center of $J$ is a distance $\geq N_{k+i_0}^{-1/2}$ away from the origin. Also, note that $I_1, I_2 \subseteq [0, N_{i_0}^{-1/2}]$.
    
    Let 
    \begin{equation*}
        F_{J} :=  \sum_{J_1 \in \cP_{j+i_0}(J \cap E_{j+i_0})} f_{\Omega_{J_1}}
    \end{equation*}
    and 
    \begin{equation*}
        G:=\sum_{J_2 \in \cP_{j+i_0}([0,N_{i+i_0}^{-1/2}] \cap E_{j+i_0})} g_{\Omega_{J_2}}.   
    \end{equation*}
    Since $f$ and $g$ are normalized in \eqref{f,gnormal}, we have
    \begin{equation*}
        \begin{split}
            &\sum_{J\in \cP_{2i+i_0}(I_1\cap E_{2i+i_0})} \norm{F_JG^2}_{L^p(\bR^2)}^2 \\
            &\leq M_{p,i_0}(j,k,2i,i)^6 \sum_{J\in \cP_{2i+i_0}(I_1\cap E_{2i+i_0})}\sum_{J_1 \in \cP_{j+i_0} (J\cap E_{j+i_0})} \norm{ f_{\Omega_{J_1}} }_{L^p(\bR^2)}^2 \\
            &\leq M_{p,i_0}(j,k,2i,i)^6 .
        \end{split}
    \end{equation*}
    Therefore, we have \eqref{keyest1nrm}, if we prove
    \begin{equation}\label{keyest1FG}
    \begin{split}
        &\norm{\sum_{J\in \cP_{2i+i_0}(I_1\cap E_{2i+i_0})} F_JG^2}_{L^p(\bR^2)}\\
        &\lesssim_{p} N_{k+i_0,i_0}^{1/2} (\tilde{C}_p B)^{i} B^{i_0} \left( \sum_{J\in \cP_{2i+i_0}(I_1\cap E_{2i+i_0})} \norm{F_JG^2}_{L^p(\bR^2)}^2 \right)^{1/2}.
    \end{split}
    \end{equation}
    We claim that for any fixed $x_1 \in \bR$,
    \begin{equation}\label{Lemsecdec}
    \begin{split}
        \int_{\bR} \bigg|\sum_{J \in \cP_{2i+i_0}(I_1 \cap E_{2i+i_0} )} & F_{J}(x_1,x_2) G(x_1,x_2)^2 \bigg|^p dx_2 \\
        &\lesssim_{p} D_p(E_{2i+i_0,i+i_0})^p \left( \frac{ N_{k+i_0}^{1/2}N_{2i+i_0}^{1/2}}{N_{i+i_0}} \right)^p \\
        &\qquad \bigg[ \sum_{J \in \cP_{2i+i_0}(I_1 \cap E_{2i+i_0} )} (\int_{\bR} |F_{J}(x_1,x_2)G(x_1,x_2)^2|^p dx_2 )^{2/p}\bigg]^{p/2}.
    \end{split}
    \end{equation}
    Assuming the claim \eqref{Lemsecdec} and using Lemma \ref{DecE_j,i}, we obtain $D_p(E_{2i+i_0,i+i_0}) \lesssim_p \widetilde{C}_p^{i}$ and by \eqref{nvcond4}, 
    \begin{equation} \label{Ndoublewi_0}
        \frac{N_{2i+i_0}^{1/2}N_{i_0}^{1/2}}{N_{i+i_0}} \leq \frac{N_{2i+2i_0}^{1/2} }{N_{i+i_0}} \leq B^{i+i_0}.
    \end{equation}
    Since $p \geq 2$, we apply Minkowski's inequality to \eqref{Lemsecdec} and obtain \eqref{keyest1FG}.\\
    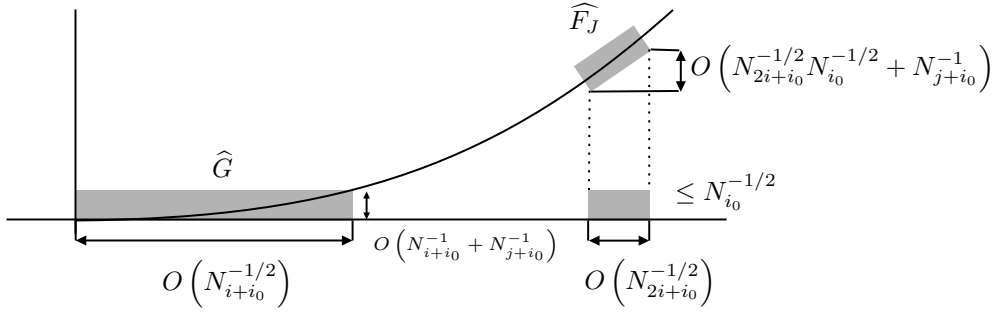
\begin{figure}
        \centering

\tikzset{every picture/.style={line width=0.75pt}} %set default line width to 0.75pt        

\begin{tikzpicture}[x=0.75pt,y=0.75pt,yscale=-1,xscale=1]
%uncomment if require: \path (0,174); %set diagram left start at 0, and has height of 174

%Straight Lines [id:da586928509724608] 
\draw    (163.45,14.87) -- (163.45,128.13) ;
\draw    (491.83,120.88) -- (128.7,120.88) ;
%Shape: Rectangle [id:dp9467286789004068] 
\draw  [color={rgb, 255:red, 0; green, 0; blue, 0 }  ,draw opacity=0 ][fill={rgb, 255:red, 0; green, 0; blue, 0 }  ,fill opacity=0.3 ] (163.54,121.37) -- (303.2,121.37) -- (303.2,105.92) -- (163.54,105.92) -- cycle ;
%Straight Lines [id:da06681019061519122] 
\draw    (163.54,121.37) -- (163.54,131.48) ;
\draw    (303.2,121.37) -- (303.2,131.48) ;
\draw    (422,121.44) -- (422,131.55) ;
\draw    (452.81,121.44) -- (452.81,131.55) ;
\draw    (166.54,131.48) -- (300.2,131.48) ;
\draw [shift={(303.2,131.48)}, rotate = 180] [fill={rgb, 255:red, 0; green, 0; blue, 0 }  ][line width=0.08]  [draw opacity=0] (5.36,-2.57) -- (0,0) -- (5.36,2.57) -- cycle    ;
\draw [shift={(163.54,131.48)}, rotate = 0] [fill={rgb, 255:red, 0; green, 0; blue, 0 }  ][line width=0.08]  [draw opacity=0] (5.36,-2.57) -- (0,0) -- (5.36,2.57) -- cycle    ;
%Straight Lines [id:da5032049979871023] 
\draw    (425.4,131.55) -- (449.81,131.55) ;
\draw [shift={(452.81,131.55)}, rotate = 180] [fill={rgb, 255:red, 0; green, 0; blue, 0 }  ][line width=0.08]  [draw opacity=0] (5.36,-2.57) -- (0,0) -- (5.36,2.57) -- cycle    ;
\draw [shift={(422.4,131.55)}, rotate = 0] [fill={rgb, 255:red, 0; green, 0; blue, 0 }  ][line width=0.08]  [draw opacity=0] (5.36,-2.57) -- (0,0) -- (5.36,2.57) -- cycle    ;
\draw    (453.54,35.06) -- (468,35.52) ;
\draw    (423.49,56.33) -- (467.63,56.01) ;
\draw    (468.73,53.47) -- (468.73,38.52) ;
\draw [shift={(468.73,35.52)}, rotate = 90] [fill={rgb, 255:red, 0; green, 0; blue, 0 }  ][line width=0.08]  [draw opacity=0] (5.36,-2.57) -- (0,0) -- (5.36,2.57) -- cycle    ;
\draw [shift={(468.73,56.47)}, rotate = 270] [fill={rgb, 255:red, 0; green, 0; blue, 0 }  ][line width=0.08]  [draw opacity=0] (5.36,-2.57) -- (0,0) -- (5.36,2.57) -- cycle    ;
\draw    (310.26,117.88) -- (310.26,109.72) ;
\draw [shift={(310.26,106.72)}, rotate = 90] [fill={rgb, 255:red, 0; green, 0; blue, 0 }  ][line width=0.08]  [draw opacity=0] (3.57,-1.72) -- (0,0) -- (3.57,1.72) -- cycle    ;
\draw [shift={(310.26,120.88)}, rotate = 270] [fill={rgb, 255:red, 0; green, 0; blue, 0 }  ][line width=0.08]  [draw opacity=0] (3.57,-1.72) -- (0,0) -- (3.57,1.72) -- cycle    ;
%Curve Lines [id:da4693297786053783] 
\draw    (163.54,121.37) .. controls (298.83,120.39) and (377.83,93.05) .. (464.8,15.12) ;
%Straight Lines [id:da3895038539951583] 
\draw  [dash pattern={on 0.84pt off 2.51pt}]  (422.76,55.83) -- (422.4,105.99) ;
\draw  [dash pattern={on 0.84pt off 2.51pt}]  (452.8,35.06) -- (452.8,105.99) ;
%Shape: Rectangle [id:dp8083391863491769] 
\draw  [color={rgb, 255:red, 0; green, 0; blue, 0 }  ,draw opacity=0 ][fill={rgb, 255:red, 0; green, 0; blue, 0 }  ,fill opacity=0.3 ] (422,121.44) -- (453.08,121.44) -- (453.08,105.99) -- (422,105.99) -- cycle ;
\draw  [color={rgb, 255:red, 0; green, 0; blue, 0 }  ,draw opacity=0 ][fill={rgb, 255:red, 0; green, 0; blue, 0 }  ,fill opacity=0.3 ] (423.49,56.33) -- (453.67,35.67) -- (444.95,22.92) -- (414.77,43.58) -- cycle ;

% Text Node
\draw (205.72,139.55) node [anchor=north west][inner sep=0.75pt]  [font=\small]  {$O\left( N_{i+i_{0}}^{-1/2}\right)$};
\draw (418.67,139.34) node [anchor=north west][inner sep=0.75pt]  [font=\small]  {$O\left( N_{2i+i_{0}}^{-1/2}\right)$};
\draw (464.47,95.54) node [anchor=north west][inner sep=0.75pt]  [font=\small]  {$\leq N_{i_{0}}^{-1/2}$};
\draw (472.17,31.14) node [anchor=north west][inner sep=0.75pt]  [font=\small]  {$O\left( N_{2i+i_{0}}^{-1/2} N_{i_{0}}^{-1/2} +N_{j+i_{0}}^{-1}\right)$};
\draw (230.82,84.89) node [anchor=north west][inner sep=0.75pt]  [font=\small]  {$\widehat{G}$};
\draw (410.8,10.09) node [anchor=north west][inner sep=0.75pt]  [font=\small]  {$\widehat{F_{J}}$};
\draw (312.26,124.28) node [anchor=north west][inner sep=0.75pt]  [font=\tiny]  {$O\left( N_{i+i_{0}}^{-1} +N_{j+i_{0}}^{-1}\right)$};

\end{tikzpicture}

        \caption{The supports of $\widehat{F_J}$ and $\widehat{G}$}
        \label{fig2}
    \end{figure}
    Now we prove the claim. Since $J \subseteq [0,N_{i_0}^{-1/2}]$, $\widehat{F}_{J}$ is supported in the horizontal strip
    \begin{equation*}
        \{(\xi_1, \xi_2) : \xi_2 = \gamma_{J}^2 + O(N_{2i+i_0}^{-1/2} N_{i_0}^{-1/2}+N_{j+i_0}^{-1}) \}
    \end{equation*}
    where $\gamma_{J}$ is the center of $J$ and since $J_2 \subseteq [0, N_{i+i_0}^{-1/2}]$, $\widehat{G}$ is supported on an $O(N_{i+i_0}^{-1/2}) \times O(N_{i+i_0}^{-1}+N_{j+i_0}^{-1})$ rectangle centered at the origin, see Figure \ref{fig2}. Since $2i \leq j$, we have $N_{j+i_0}^{-1} \leq N_{i+i_0}^{-1}$ and also note that $N_{2i+i_0}^{-1/2} N_{i_0}^{-1/2} \leq N_{i+i_0}^{-1}$. Therefore, $\widehat{F_{J}G^2}$ is supported in the horizontal strip
    \begin{equation*}
        \{(\xi_1, \xi_2 ) : \xi_2 = \gamma_{J}^2 + O(N_{i+i_0}^{-1}) \}.
    \end{equation*}
    The Fourier transform of $F_{J}G^2$ in $y$ for fixed $x$ is also supported on an interval of length $O(N_{i+i_0}^{-1})$ centered at $\gamma_{J}^2$.
    
    Let $c \approx N_{k+i_0,i_0}^{1/2}N_{2i+i_0}^{1/2}N_{i+i_0}^{-1} $ and $cJ$ be the interval of length $c|J|$ which has the same center with $J$. Since $I_1$ is the interval of length $N_{i+i_0}^{-1/2}$, we have
    \begin{equation}\label{LemLindecsca}
    \begin{split}
        &\norm{\sum_{J \in \cP_{2i+i_0}(I_1 \cap E_{2i+i_0})} f_{cJ}}_{L^p(\bR)} \lesssim \\
        &D_p(E_{2i+i_0,i+i_0}) \left( \frac{N_{k+i_0}^{1/2}N_{2i+i_0}^{1/2}}{N_{i+i_0}}\right) \left(\sum_{J \in \cP_{2i+i_0}(I_1 \cap E_{2i+i_0})} \norm{f_{cJ}}_{L^p(\bR)}^2  \right)^{1/2}.
    \end{split}
    \end{equation}
    Let $\gamma_m$ be the left endpoint of $I_1 $ and let us consider 
    \begin{equation*}
        T(x) = (2\gamma_m+N_{i+i_0}^{-1/2})(x-\gamma_m)+  \gamma_m^2.
    \end{equation*}
    Note that $T(x)$ is an equation of a line which passes through $(\gamma_m, \gamma_m^2)$ and $(\gamma_m+N_{i+i_0}^{-1/2},(\gamma_m+N_{i+i_0}^{-1/2})^2) $. Since $\gamma_m \geq N_{k+i_0}^{-1/2}$ and $|\gamma_J - \gamma_m| \leq |I_1| = N_{i+i_0}^{-1/2}$, 
    \begin{equation*}
        \begin{split}
            T^{-1}(\gamma_J^2 + O(N_{i+i_0}^{-1})) &=\gamma_J + \frac{(\gamma_J - \gamma_m) (\gamma_J - \gamma_m - N_{i+i_0}^{-1/2})}{2\gamma_m + N_{i+i_0}^{-1/2}}+O(N_{k+i_0}^{1/2} N_{i+i_0}^{-1})\\
            &= \gamma_J + O(N_{k+i_0}^{1/2}N_{i+i_0}^{-1}).
        \end{split}
    \end{equation*}
    
    We use $|J| = N_{2i+i_0}^{-1/2}$ and obtain that $\gamma_J^2+ O(N_{i+i_0}^{-1}) \subseteq T(cJ) $ if $c$ is sufficiently large while still comparable to $N_{k+i_0,i_0}^{1/2}N_{2i+i_0}^{1/2}N_{i+i_0}^{-1}$. The intervals $cJ$ in \eqref{LemLindecsca} can be replaced by $T(cJ)$ and we obtain \eqref{Lemsecdec}.
    
    Now, we turn to the proof of \eqref{keyest}. Similarly, we fix an arbitrary $a \in A_{i_0}$. We only need to consider $f $ and $g$ such that 
    \begin{equation*}
        \sum_{J_1\in \cP_{j+i_0( I_1 \cap E_{j+i_0}) }} \norm{f_{\Omega_{J_1}}}_{L^{3p}(\bR^2)}^2 =\sum_{J_2\in \cP_{j+i_0( I_2 \cap E_{j+i_0}) }} \norm{g_{\Omega_{J_2}}}_{L^{3p}(\bR^2)}^2 = 1.
    \end{equation*}
    We can assume that $I_2:=[0, N_{2i+i_0}^{-1/2}]$ is on the left of an interval $I_1$ of length $N_{i+i_0}^{-1/2}$ and $I_1,I_2 \subseteq [0,N_{i_0}^{-1/2}]$. For each $J\in P_{4i + i_0}(I_1 \cap E_{4i+i_0})$, the center of $J$ is a distance $\geq N_{k+i_0}^{-1/2}$ away from the origin. Let 
    \begin{equation*}
        F_{J} :=  \sum_{J_1 \in \cP_{j+i_0}(J \cap E_{j+i_0})} f_{\Omega_{J_1}}
    \end{equation*}
    and 
    \begin{equation*}
        G:=\sum_{J_2 \in \cP_{j+i_0}([0,N_{2i+i_0}^{-1/2}] \cap E_{j+i_0})} g_{\Omega_{J_2}}.   
    \end{equation*}
    By the same argument in the proof of \eqref{keyest1}, it suffices to prove that, for any fixed $x_1 \in \bR$,
    \begin{equation}\label{Lemsecdec1}
    \begin{split}
        &\int_{\bR} \bigg|\sum_{J \in \cP_{4i+i_0}(I_1 \cap E_{4i+i_0} )}  F_{J}(x_1,x_2) G(x_1,x_2)^2 \bigg|^p dx_2 \\
        &\lesssim_{p}  D_p(E_{4i+i_0,2i+i_0})^p D_p(E_{2i+i_0,i+i_0})^p \left( \frac{N_{k+i_0}^{1/2}N_{4i+i_0}^{1/2}}{N_{2i+i_0}} \right)^p\left( \frac{N_{k+i_0}^{1/2}N_{2i+i_0}^{1/2}}{N_{i+i_0}} \right)^p \\
        &\qquad \bigg[ \sum_{J \in \cP_{4i+i_0}(I_1 \cap E_{4i+i_0} )} (\int_{\bR} |F_{J}(x_1,x_2)G(x_1,x_2)^2|^p dx_2 )^{2/p}\bigg]^{p/2},
    \end{split}
    \end{equation}
    which corresponds to \eqref{Lemsecdec}. By Lemma \ref{DecE_j,i} and using \eqref{Ndoublewi_0} twice, we obtain \eqref{keyest}.
    
    We now prove \eqref{Lemsecdec1}. For $J_0 \in \cP_{2i+i_0}(I_1 \cap E_{2i+i_0})$, let
    \begin{equation*}
        F_{J_0} = \sum_{J \in \cP_{4i+i_0}(J_0 \cap E_{4i+i_0})} F_J,
    \end{equation*}
    so that
    \begin{equation*}
        \sum_{J \in \cP_{4i+i_0}(I_1 \cap E_{4i+i_0})} F_J = \sum_{J_0 \in \cP_{2i+i_0}(I_1 \cap E_{2i+i_0})} F_{J_0}.
    \end{equation*}
    Since $J_0 \subseteq [0, N_{i_0}^{-1/2}]$, $\widehat{F}_{J_0}$ is supported in the horizontal strip
    \begin{equation*}
        \{(\xi_1, \xi_2) : \xi_2 = \gamma_{J_0}^2 + O(N_{2i+i_0}^{-1/2}N_{i_0}^{-1/2} +N_{j+i_0}^{-1}) \}
    \end{equation*}
    where $\gamma_{J_0}$ is the center of $J_0$ and since $J_2 \subseteq [0, N_{2i+i_0}^{-1/2}]$, $\widehat{G}$ is supported on an $O(N_{2i+i_0}^{-1/2}) \times O(N_{2i+i_0}^{-1}+N_{j+i_0}^{-1})$ rectangle. Then, $\widehat{F_{J_0}G^2}$ is supported in the horizontal strip
    \begin{equation*}
        \{(\xi_1, \xi_2 ) : \xi_2 = \gamma_{J_0}^2 + O(N_{i+i_0}^{-1}) \}.
    \end{equation*}
    By the same argument in the proof of \eqref{Lemsecdec}, we obtain that
    \begin{equation}\label{Lemsecdec2}
    \begin{split}
        &\int_{\bR} \bigg|\sum_{J \in \cP_{4i+i_0}(I_1 \cap E_{j+i_0} )} F_{J}(x_1,x_2) G(x_1,x_2)^2 \bigg|^p dx_2 \\
        &=\int_{\bR} \bigg|\sum_{J_0 \in \cP_{2i+i_0}(I_1 \cap E_{2i+i_0} )} F_{J_0}(x_1,x_2) G(x_1,x_2)^2 \bigg|^p dx_2 \\
        &\lesssim_{p}  D_p(E_{2i+i_0,i+i_0})^p \left( \frac{N_{k+i_0}^{1/2}N_{2i+i_0}^{1/2}}{N_{i+i_0}} \right)^p\\
        &\qquad\bigg[ \sum_{J_0 \in \cP_{2i+i_0}(I_1 \cap E_{2i+i_0} )} (\int_{\bR} |F_{J_0}(x_1,x_2)G(x_1,x_2)^2|^p dx_2 )^{2/p}\bigg]^{p/2}.
    \end{split}
    \end{equation}
    Similarly, $\widehat{F}_{J}$ is supported in the horizontal strip
    \begin{equation*}
        \{(\xi_1, \xi_2) : \xi_2 = \gamma_{J}^2 + O(N_{4i+i_0}^{-1/2}N_{i_0}^{-1/2} +N_{j+i_0}^{-1}) \}
    \end{equation*}
    where $\gamma_{J}$ is the center of $J$. Since $4i \leq j$, Fourier transform of $F_{J}G^2$ in $y$ for fixed $x$ is supported in the horizontal strip
    \begin{equation*}
        \{ (\xi_1,\xi_2): \xi_2 = \gamma_J^2 + O(N_{2i+i_0}^{-1}) \}.
    \end{equation*}
    We use the argument in the proof of \eqref{Lemsecdec} with $i $ replaced by $2i$ and obtain
    \begin{equation}\label{Lemsecdec3}
    \begin{split}
        &\int_{\bR} \bigg|\sum_{J \in \cP_{4i+i_0}(J_0 \cap E_{4i+i_0} )} F_{J}(x_1,x_2) G(x_1,x_2)^2 \bigg|^p dx_2 \\
        &\lesssim_{p} D_p(E_{4i+i_0,2i+i_0})^p \left( \frac{N_{k+i_0}^{1/2}N_{4i+i_0}^{1/2}}{N_{2i+i_0}} \right)^p\\
        &\qquad \bigg[ \sum_{J \in \cP_{4i+i_0}(J_0  \cap E_{4i+i_0})} (\int_{\bR} |F_{J}(x_1,x_2)G(x_1,x_2)^2|^p dx_2 )^{2/p}\bigg]^{p/2}.
    \end{split}
    \end{equation}
    Combining \eqref{Lemsecdec2} and \eqref{Lemsecdec3}, we get \eqref{Lemsecdec1}.
\end{proof}

\begin{lem}\label{Decsplit}
    Let $k \leq i_2 \leq i_1 \leq j$. Then,
    \begin{equation*}
        M_{p,i_0}(j,k,i_1,i_2) \leq M_{p,i_0}(j,k,i_2,i_1)^{1/2} D_{3p}(P_{j+i_0,i_2+i_0})^{1/2}.
    \end{equation*}
\end{lem}
\begin{proof}
    As in \cite{CDGJLM22}, it follows from  $\int F^{p} G^{2p} dx \leq (\int F^{2p} G^{p})^{1/2} (\int G^{3p})^{1/2}$ where $F$ and $G$ are nonnegative functions and the definition of $M_{p,i_0}(j,k,i_1,i_2)$.
\end{proof}
We now can prove Proposition \ref{DecP_k}.
\begin{proof}[Proof of Proposition \ref{DecP_k}]
    Let $p=2/\alpha$ and assume that $\lambda$ is the smallest exponent such that
    \begin{equation}\label{Declambda}
        D_{3p}(P_{j+i_0,i_0}) \lesssim_{p,\epsilon} N_{j+i_0}^\epsilon N_{j+i_0,i_0}^{\lambda}
    \end{equation}
    for all $0 \leq i_0 ,j$ and for any $0 < \epsilon < 1$. Then, it suffices to show that $\lambda=0$. We can run the iteration as in \cite{CDGJLM22}. But, we should also consider that $n_j$ is nondecreasing. There is a trivial estimate $D_{3p}(P_{j+i_0,i_0}) \lesssim N_{j+i_0,i_0}^{1/4}$, so
    we assume that $0 < \lambda \leq 1/4$ toward a contradiction.
    
    By \eqref{nvcond2}, for any $i_0 \geq 0$ we have
    \begin{equation}\label{tildeC_pB}
        (\widetilde{C}_pB)^{i/3} \lesssim_{p} N_{i} \leq N_{i+i_0,i_0}.
    \end{equation}
    Combining Lemma \ref{Decbil}, Lemma \ref{Deckeylem} and \eqref{tildeC_pB}, we obtain that if $j \geq 2i$, then
    \begin{equation}\label{prop3.5st1}
    \begin{split}
        D_{3p}(P_{j+i_0,i_0})&\leq D_{3p}(P_{j+i_0,i+i_0})+N_{i+i_0,i_0}^{O(1)}  (\widetilde{C}_p B)^{i/3} B^{i_0/3} M_{p,i_0}(j,i,2i,i)\\
        &\leq D_{3p}(P_{j+i_0,i+i_0})+N_{i+i_0,i_0}^{O(1)} B^{i_0/3} M_{p,i_0}(j,i,2i,i).
    \end{split}
    \end{equation}
    Now we need an estimate for $M_{p,i_0}(j,i,2i,i)$. By Lemma \ref{Deckeylem} and Lemma \ref{Decsplit}, for any positive integer $a$ such that $1 \leq a \leq \frac{j}{4i}$, we obtain that
    \begin{equation}\label{iter}
    \begin{split}
        M_{p,i_0}(j,i,2ai,ai) &\leq  M_{p,i_0}(j,i,ai,2ai)^{1/2} D_{3p}(P_{j+i_0,ai+i_0})^{1/2}\\
        &\lesssim_{p} N_{i+i_0,i_0}^{O(1)}(\widetilde{C}_p B)^{ai/2} B^{i_0/3} M_{p,i_0}(j,i,4ai,2ai)^{1/2} D_{3p}(P_{j+i_0,ai+i_0})^{1/2} .
    \end{split}
    \end{equation}
    First, we assume that $j = 2^{k+1}i$ where $k $ is a sufficiently integer to be determined later. By iterating \eqref{iter} and using \eqref{Declambda} with $\epsilon$ replaced by $\epsilon/2$, we have
    \begin{equation*}
        \begin{split}
            M_{p,i_0}(j,i,2i,i)\lesssim_{p,\epsilon} &N_{i+i_0,i_0}^{O(1)} (\widetilde{C}_pB)^{ki/2} B^{2i_0/3}\\
            &N_{j+i_0}^{(\lambda+\epsilon/2)(1-1/2^{k})} (N_{i+i_0} N_{2i+i_0}^{1/2} \cdots N_{2^{k-1}i+i_0}^{1/2^{k-1}})^{-\lambda/2} M_{p,i_0}(j,i,2^{k+1}i, 2^k i)^{1/2^{k}}.
        \end{split}
    \end{equation*}
    We use \eqref{multi-lin_dec} and \eqref{Declambda} with $\epsilon $ replaced by $\epsilon /2$ and we get
    \begin{equation*}
        M_{p,i_0}(j,i,2^{k+1}i, 2^ki) \leq D_{3p}(P_{j+i_0,2^{k}i+i_0}) \lesssim_{p,\epsilon}  N_{j+i_0}^{\epsilon/2} N_{j+i_0,{2^ki+i_0}}^{\lambda}.
    \end{equation*}
    Therefore, we arrive at
    \begin{equation*}
    \begin{split}
        M_{p,i_0}(j,i,2i,i)\lesssim_{p,\epsilon} & N_{j+i_0}^{\epsilon/2} N_{j+i_0,i_0}^\lambda N_{i+i_0,i_0}^{O(1)}  (\widetilde{C}_pB)^{ki/2} B^{2i_0/3}\\
        &\left(\frac{N_{i_0}}{N_{i+i_0}^{1/2} N_{2i+i_0}^{1/4} \cdots N_{2^{k-1}i+i_0}^{1/2^k} N_{2^ki+i_0}^{1/2^k}  }\right)^{\lambda}.
    \end{split}
    \end{equation*}
    Since $n_j$ is nondecreasing, we obtain that for any $s \geq 0$,
    \begin{equation*}
        \left(\frac{N_{i_0}}{N_{2^s i+i_0}}\right)^{1/2^s} \leq \frac{N_{i_0}}{N_{i+i_0}}.
    \end{equation*}
    Therefore, we have
    \begin{equation*}
        \left(\frac{N_{i_0}}{N_{i+i_0}^{1/2} N_{2i+i_0}^{1/4} \cdots N_{2^{k-1}i+i_0}^{1/2^{k}} N_{2^ki+i_0}^{1/ 2^k} }\right) \leq N_{i+i_0,i_0}^{-(k+1)/2} ,
    \end{equation*}
    and it leads to
    \begin{equation*}
    \begin{split}
        M_p(j,i,2i,i)&\lesssim_{p,\epsilon}  N_{j+i_0}^{\epsilon/2} N_{j+i_0,i_0}^\lambda N_{i+i_0,i_0}^{O(1)-\lambda (k+1)/2}   (\widetilde{C}_pB)^{ki/2} B^{2i_0/3}.
    \end{split}
    \end{equation*}
    Since $j=2^{k+1}i$, $ki \leq j$ for any $k\geq 0$. By using \eqref{nvcond2}, we have
    \begin{equation}\label{C_pBki}
        (\widetilde{C_p}B)^{ki/2} \leq (\widetilde{C_p}B)^{j/2} \lesssim_{p,\epsilon} N_j^{\epsilon/4} \leq N_{j+i_0}^{\epsilon/4}.
    \end{equation}
    Since $\lambda \leq 1/4$, $N_{j+i_0,i_0}^\lambda \leq N_{j+i_0,i+i_0}^\lambda N_{i+i_0,i_0}^{O(1)}$ and by \eqref{C_pBki}, we get
    \begin{equation}\label{M_p(j,i,2i,i)}
        M_p(j,i,2i,i) \lesssim_{p,\epsilon}  N_{j+i_0}^{3\epsilon/4} N_{j+i_0,i+i_0}^\lambda N_{i+i_0,i_0}^{O(1)-\lambda (k+1)/2}  B^{2i_0/3}.
    \end{equation}
    We choose $k$ such that $O(1)-\lambda (k+1)/2 <0$ and $k \approx \lambda^{-1}$, then it follows from \eqref{Declambda} with $\epsilon$ replaced by $3\epsilon/4$, \eqref{prop3.5st1} and \eqref{M_p(j,i,2i,i)} that
    \begin{equation*}
        \begin{split}
            D_{3p}(P_{j+i_0,i_0}) &\leq D_{3p}(P_{j+i_0,i+i_0})+N_{i+i_0,i_0}^{O(1)} B^{i_0/3} M_{p,i_0}(j,i,2i,i)\\
            &\lesssim_{p,\epsilon} N_{j+i_0}^{3\epsilon/4} N_{j+i_0,i+i_0}^\lambda + N_{j+i_0}^{3\epsilon/4} N_{j+i_0,i+i_0}^\lambda N_{i+i_0,i_0}^{O(1)-\lambda (k+1)/2} B^{i_0} \\ 
            &\lesssim_{p,\epsilon} N_{j+i_0}^{3\epsilon/4} N_{j+i_0,i+i_0}^\lambda B^{i_0}.
        \end{split}
    \end{equation*}
    Since $j= 2^{k+1}i$, we use \eqref{Ndoublewi_0} at all scales $i,2i, \cdots, 2^k i$ and obtain that
    \begin{equation*}
        N_{i+i_0}^{-1} \leq B^{i+i_0}N_{2i+i_0}^{-1/2}N_{i_0}^{-1/2}  \leq \cdots \leq B^{i(k+1) +2i_0} N_{j+i_0}^{-1/2^{k+1}} N_{i_0}^{-1+1/2^{k+1}}.
    \end{equation*}
    Therefore, we have
    \begin{equation*}
        D_{3p}(P_{j+i_0,i_0}) \lesssim_{p,\epsilon} N_{j+i_0}^{3\epsilon /4}( { N_{j+i_0,i_0}})^{\lambda(1-1/2^{k+1})}  B^{i(k+1)\lambda} B^{i_0(2\lambda+1)}.
    \end{equation*}
    Since $\lambda \leq 1/4$ and $ki \leq j$, by using \eqref{nvcond2}, 
    \begin{equation*}
        B^{i(k+1)\lambda} B^{i_0(2\lambda+1)} \leq B^{2(j+i_0)} \lesssim_\epsilon N_{j+i_0}^{\epsilon/4}.
    \end{equation*}
    Now, we obtain that
    \begin{equation}\label{DecDforj2k}
        D_{3p}(P_{j+i_0,i_0}) \lesssim_{p,\epsilon} N_{j+i_0}^\epsilon N_{j+i_0,i_0}^{\lambda(1-1/2^{k+1}) }.
    \end{equation}
    
    If $2^{k+1}i \leq j \leq 2^{k+1}(i+1)$, it follows from Lemma \ref{DecAM} that
    \begin{equation*}
        D_{3p}(P_{j+i_0,i_0}) \leq D_{3p}(P_{2^{k+1}i+i_0, i_0})D_{3p}(P_{j+i_0,2^{k+1}i+i_0}).
    \end{equation*}
    By \eqref{DecDforj2k} with $\epsilon$ replaced by $\epsilon/2$ and the trivial estimate,
    \begin{equation*}
    \begin{split}
        D_{3p}(P_{j+i_0,i_0}) &\lesssim_{p,\epsilon} N_{2^{k+1}i+i_0}^{\epsilon/2} N_{2^{k+1}i+i_0,i_0}^{\lambda(1-1/2^{k+1})} N_{j+i_0,2^{k+1}i+i_0}^{1/4}\\
        &\leq N_{j+i_0}^{\epsilon/2} N_{j+i_0,i_0}^{\lambda(1-1/2^{k+1})} N_{j+i_0,2^{k+1}i+i_0}^{1/4}.
    \end{split}
    \end{equation*}
    By \eqref{nvcond3}, there exists a constant $C_{\lambda,\epsilon}$ such that
    \begin{equation*}
        n_{j+i_0} \leq n_{j+i_0+1} \leq C_{\lambda,\epsilon} N_{j+i_0}^{\epsilon / 2^{k}}.
    \end{equation*}
    Since $k \approx \lambda^{-1}$, the constant $C_{\lambda,\epsilon}$ depends on $\lambda$. Then, we have
    \begin{equation*}
        \begin{split}
             N_{j+i_0,2^{k+1}i+i_0}^{1/4} \leq (n_{j+i_0})^{2^{k-1}} \leq C_{\lambda,\epsilon}^{2^{k-1}} N_{j+i_0}^{\epsilon/2} \lesssim_{\lambda,\epsilon} N_{j+i_0}^{\epsilon/2}.
        \end{split}
    \end{equation*}
    Therefore, we obtain that for any $j,i_0 \geq 0$,
    \begin{equation*}
        D_{3p}(P_{j+i_0,i_0}) \lesssim_{p,\epsilon, \lambda } N_{j+i_0}^{\epsilon} N_{j+i_0,i_0}^{\lambda(1-1/2^{k+1})}. 
    \end{equation*}
    This contradicts the assumption that $\lambda>0$ is the smallest exponent which satisfies \eqref{Declambda}. Therefore, $\lambda =0$.
\end{proof}

\section{Local restriction estimate}\label{secLR}
We denote $B^d(R)$ by a cube of side length $R$ in $ \bR^d$ centered at the origin.
\begin{prop}\label{PropLRresult}
    Let $p= 6/\alpha$ where $0 < \alpha <1$ and $\nu$ be the measure constructed in Section \ref{secCan}. Then, we have
    \begin{equation}\label{eqLRpropRes}
        \norm{\widehat{fd\nu}}_{L^{p}(B^2(R))}\lesssim_{p,\epsilon} R^\epsilon \norm{f}_{L^2(\nu)} \qquad \forall f \in L^2(\nu).
    \end{equation}
    Equivalently, for $q$ such that $1/q + 1/p=1$, we have
    \begin{equation}\label{coreqLR}
        \norm{\widehat{f}}_{L^2(\nu)} \lesssim_{q,\epsilon} R^\epsilon \norm{f}_{L^{q} (\bR^2)}
    \end{equation}
    for any $f \in L^q(\bR^2)$ supported on $B^2(R)$.
\end{prop}
Proposition \ref{PropLRresult} follows from the following lemma.
\begin{lem}\label{lemLR}
    For $\ell \geq 2j$ and for every $a \in A_{2j}$, we define 
    \begin{equation*}
        \widehat{f}_{\Omega_a} (\xi) = \widehat{f}(\xi) \sum_{b \in A_{\ell,2j,a}} \mathbf{1}_{[0,1]\times [-1,1]} (N_{\ell}^{1/2}\xi_1-b , N_{\ell}\xi_2 -(2b+1) N_{\ell}^{1/2}\xi_1 + b^2+b).
    \end{equation*}
    For any $\sum_{a \in A_{2j} } f_{\Omega_a} \in L^2(\bR^2)$, we have the estimate
    \begin{equation}\label{eqLemLR}
    \norm{\sum_{a \in A_{2j} } f_{\Omega_a}}_{L^p(B^2(N_j))} \lesssim_{\epsilon,p} K_p(P_{2j})N_{2j}^{\epsilon+\frac{3}{2p}-\frac{\alpha}{4}} N_\ell^{-3/4} |A_\ell|^{1/2} \norm{\sum_{a \in A_{2j} } f_{\Omega_a}}_{L^2(\bR^2)}.
    \end{equation}
\end{lem}
\begin{proof}[Proof of Proposition \ref{PropLRresult}]
    Let $\eta$ be a Schwartz function on $\bR$ such that $|\eta| \geq 1$ on $[-1,1]$ and $\widehat{\eta}$ is supported on $[0,1]$. Let
    \begin{equation*}
        \widehat{F}(\xi) := f(\xi_1,\xi_1^2) \sum_{a \in A_{2j}}\sum_{b \in A_{\ell,2j,a}} \mathbf{1}_{[0,1]}(N_\ell^{1/2}\xi_1 - b)\widehat{{\eta}}_{\ell}(\xi_2-\xi_1^2) |E_\ell|^{-1}
    \end{equation*}
    where ${\eta}_{\ell}(x) ={\eta}(x/ N_\ell)$ for $x \in \bR$. Note that
    \begin{equation*}
        \sum_{a \in A_{2j}}\sum_{b \in A_{\ell,2j,a}} \mathbf{1}_{[0,1]}(N_\ell^{1/2}\xi_1 - b) = \mathbf{1}_{E_\ell}(\xi_1).    
    \end{equation*}
    If $x \in B^2(N_j)$, we have
    \begin{equation*}
        |\widehat{fd\nu_\ell}(x)| \leq  |\widehat{fd\nu_\ell}(x) \eta_\ell(-x_2)| = |F(-x)|.
    \end{equation*}
    Therefore, \eqref{eqLemLR} implies that
    \begin{equation*}
        \begin{split}
            \norm{\widehat{f d\nu_\ell} }_{L^p(B^2(N_j))} &\lesssim \norm{F}_{L^p (B^2(N_j))} \\
            & \lesssim_\epsilon  K_p(P_{2j})N_{2j}^{\epsilon + \frac{3}{2p}-\frac{\alpha}{4}} N_{\ell}^{-3/4} |A_\ell|^{1/2} \norm{{F}}_{L^2(\bR^2)}\\
            & =  K_p(P_{2j})N_{2j}^{\epsilon + \frac{3}{2p}-\frac{\alpha}{4}} N_{\ell}^{-3/4} |A_\ell|^{1/2} \norm{\widehat{F}}_{L^2(\bR^2)}.
        \end{split}
    \end{equation*}
    Since $|E_{\ell}| = |A_\ell| N_\ell^{-1/2}$, we get $\norm{\widehat{F}}_{L^2(\bR^2)} \lesssim N_{\ell}^{3/4} |A_\ell|^{-1/2} \norm{f}_{L^2(\nu_\ell)}$.
    
    Since $p =6/\alpha$, it follows from Proposition \ref{DecP_k} that 
    \begin{equation*}
    \begin{split}
        \norm{\widehat{f d\nu_\ell} }_{L^{p}(B^2(N_j))} &\lesssim_\epsilon K_{p}(P_{2j}) N_{2j}^\epsilon \norm{f}_{L^2(\nu_\ell)}\\
        &\lesssim_\epsilon N_{2j}^{2\epsilon}\norm{f}_{L^2(\nu_\ell)}.
    \end{split}
    \end{equation*}
    If $N_{j-1} \leq R \leq N_{j}$, conditions \eqref{nvcond2}, \eqref{nvcond3} and \eqref{nvcond4} imply that 
    \begin{equation*}
        N_{2j} \leq B^{2j} N_j^2 =B^{2j}n_{j}^2 N_{j-1}^2 \lesssim_\epsilon N_{j-1}^{2+ \epsilon}.  
    \end{equation*} 
    Therefore, 
    \begin{equation*}
    \begin{split}
        \norm{\widehat{f d\nu_\ell} }_{L^{p}(B^2(R))} &\leq \norm{\widehat{f d\nu_\ell} }_{L^{p}(B^2(N_{j}))}\\
        &\lesssim N_{2j}^{2\epsilon} \norm{f}_{L^2(\nu_\ell)}\\
        &\lesssim_\epsilon  N_{j-1}^{O(\epsilon)}  \norm{f}_{L^2(\nu_\ell)}\\
        &\lesssim_\epsilon R^{O(\epsilon)} \norm{f}_{L^2(\nu_\ell)}.
    \end{split}
    \end{equation*}
    When we take the limit $\ell \rightarrow \infty$, we get \eqref{eqLRpropRes} and by duality, \eqref{coreqLR} follows.
\end{proof}
Now let us turn to the proof of Lemma \ref{lemLR}.
\begin{proof}[Proof of Lemma \ref{lemLR}]
     Let $\eta$ be a Schwartz function on $\bR^2$ such that $|\eta| \geq 1$ on $[-1,1]^2$ and $\widehat{\eta}$ is supported on $[0,1]^2$ and write ${\eta}_{2j}(x) = \eta( x / N_{2j}^{1/2}) $. Since $N_j \leq N_{2j}^{1/2}$, we have
    \begin{equation*}
    \begin{split}
        \norm{\sum_{a \in A_{2j}} f_{\Omega_a}}_{L^p(B^2(N_j))} &\leq \norm{\sum_{a \in A_{2j}} f_{\Omega_a} {\eta_{2j}} }_{L^p(\bR^2) }\\
        &=\sup_{\norm{g}_{L^q (\bR^2)} \leq 1 } \left| \int \sum_{a \in A_{2j}} f_{\Omega_a }(x) {\eta}_{2j}(x) g(x)  dx \right| \\
        &=\sup_{\norm{g}_{L^q (\bR^2)} \leq 1 }\left| \int \sum_{a \in A_{2j}} \widehat{f}_{\Omega_a} \ast \widehat{{\eta}}_{2j} (\xi) \widehat{g}(\xi) d\xi \right|
    \end{split}
    \end{equation*}
    where $1/p+1/q=1$.
    
    The function $\widehat{f}_{\Omega_a}$ is supported on a rectangle of dimensions $O(N_{2j}^{-1/2})\times  O(N_{2j}^{-1})$ centered at $((a+1/2) N_{2j}^{-1/2}, (a+1/2)^2 N_{2j}^{-1})$ and the direction that it is pointing depends on $a$. The function $\widehat{{\eta}}_{2j}$ is supported on a square with side length $O(N_{2j}^{-1/2})$ centered at the origin. Therefore, $\widehat{f}_{\Omega_a} \ast \widehat{{\eta}}_{2j} (\xi)$ is supported on a square with side length $O(N_{2j}^{-1/2})$ centered at $(aN_{2j}^{-1/2}, a^2 N_{2j}^{-1})$.
    
    For $a \in A_{2j}$ and $c=(c_1,c_2) \in \bZ^2$, let us consider characteristic functions
    \begin{equation*}
        \mathbf{1}_{Q_{a,c}}(\xi) = \mathbf{1}_{[0,1]^2} (N_{2j}^{1/2} \xi_1 - (a+c_1), N_{2j}^{1/2} \xi_2 -(a^2N_{2j}^{-1/2}+c_2) ).
    \end{equation*}
    Then, we have
    \begin{equation*}
            \norm{\sum_{a \in A_{2j}} f_{\Omega_a}}_{L^p(B^2(N_j))} \leq  \sup_{\norm{g}_{L^q (\bR^2)} \leq 1 }\sum_{|c| \lesssim 1}\left| \int \sum_{a \in A_{2j}} \widehat{f}_{\Omega_a} \ast \widehat{{\eta}}_{2j} (\xi) \widehat{g}(\xi) \mathbf{1}_{Q_{a,c}}(\xi) d\xi \right|.
    \end{equation*}
    By letting
    \begin{equation}\label{chvar}
        z_1 = N_{2j}^{1/2} \xi_1- (a+c_1) \qquad \mathrm{and}    \qquad z_2 = N_{2j}^{1/2} \xi_2 -(a^2 N_{2j}^{-1/2 }+c_2),
    \end{equation}
    we obtain
    \begin{equation*}
        \begin{split}
            \sup_{\norm{g}_{L^q (\bR^2)} \leq 1 } \sum_{|c| \lesssim 1} \Bigg|\int_{[0,1]^2} \sum_{a \in A_{2j}} &\widehat{f}_{\Omega_a} \ast \widehat{{\eta}}_{2j} (N_{2j}^{-1/2}( z_1 + a+c_1), N_{2j}^{-1/2} (z_2+c_2) + a^2 N_{2j}^{-1}  )\\
            &\widehat{g}(N_{2j}^{-1/2} (z_1 + a+c_1), N_{2j}^{-1/2}( z_2+c_2) + a^2 N_{2j}^{-1}  ) dz N_{2j}^{-1 } \Bigg| .
        \end{split}
    \end{equation*}
    Let $a'=(a_1',a_2')\in \bZ^2$ and $a'\cdot u = a_1' u_1 + a_2'u_2$. Then, we have 
    \begin{equation}\label{plugin}
        \sum_{|a_1'| \lesssim N_{2j}^{1/2} , |a_2'| \lesssim N_{2j}} \int_{[0,1]^2} e(au_1 +a^2u_2) e(-a'\cdot u)du=1   
    \end{equation}
    if and only if $a_1'=a$ and $a_2' =a^2$ for $|a| \lesssim N_{2j}^{1/2}$.
    
    Let
    \begin{equation*}
        F_c(u,z) :=\sum_{a \in A_{2j}} \widehat{f}_{\Omega_a} \ast \widehat{{\eta}}_{2j} (N_{2j}^{-1/2}( z_1 + a+c_1), N_{2j}^{-1/2} (z_2+c_2) + a^2 N_{2j}^{-1}  )e(au_1 + a^2u_2)
    \end{equation*}
    and 
    \begin{equation*}
        G_c(u,z) := \sum_{|a_1'| \lesssim N_{2j}^{1/2} , |a_2'| \lesssim N_{2j}}\widehat{g}(N_{2j}^{-1/2} (z_1 + a_1' +c_1), N_{2j}^{-1/2} (z_2+c_2) + a_2' N_{2j}^{-1}) e(-a'\cdot u).
    \end{equation*}
    Also, we denote the mixed norm of $f$ by
    \begin{equation*}
        \norm{f}_{L^{p_1}_u L^{p_2}_z([0,1]^2)} = \bigg[ \int_{[0,1]^2} \left( \int_{[0,1]^2} |f(u,z)|^{p_1} du\right)^{p_2/p_1} dz \bigg]^{1/p_2}.
    \end{equation*}
    By \eqref{plugin} and H\"{o}lder's inequality, we have
    \begin{equation}\label{spltF_cG_c}
    \begin{split}
        \norm{\sum_{a \in A_{2j}} f_{\Omega_a} }_{L^p(B^2(N_j))} &\leq \sup_{\norm{g}_{L^q(\bR^2) \leq 1}} \sum_{ |c| \lesssim 1} \left|\int_{[0,1]^2}\int_{[0,1]^2} F_c(u,z) G_c(u,z) dudz\right| N_{2j}^{-1 }\\
        &\leq \sup_{\norm{g}_{L^q(\bR^2) \leq 1}} \sum_{ |c| \lesssim 1} \norm{F_c(u,z)}_{L^p_u L^2_z([0,1]^2)} \norm{G_c(u,z)}_{L^q_uL^2_z([0,1]^2)} N_{2j}^{-1 }.
    \end{split}
    \end{equation}
    Since $F_c(u,z) $ is a Fourier series with respect to $u $ variable, the definition of $K_p(P_{2j})$ and \eqref{chvar} implies that 
    \begin{equation}\label{estF_c}
    \begin{split}
        \norm{F_c}_{L^p_u L^2_z([0,1]^2)} &\leq K_p(P_{2j}) \norm{F_c}_{L^2_u L^2_z([0,1]^2)} \\
        &\leq  K_p(P_{2j}) N_{2j}^{1/2} \left( \int \sum_{a \in A_{2j}} |\widehat{f}_{\Omega_a} \ast \widehat{\eta}_{2j}(\xi)|^2 d\xi  \right)^{1/2}. 
    \end{split}
    \end{equation}
    
    In order to obtain an estimate for $\norm{G_c}_{L_u^q L_z^2([0,1]^2)}$, we consider $G_c$ as a linear operator $T: L^q (\bR^2) \rightarrow L_u^qL_z^2([0,1]^2)$ acting on $g$, which is defined by
    \begin{equation*}
    \begin{split}
        Tg(u,z) =\sum_{|a_1'| \lesssim N_{2j}^{1/2} , |a_2'| \lesssim N_{2j}} &\int g(x) e(N_{2j}^{-1/2} z_1x_1 + N_{2j}^{-1/2}z_2x_2 )\\
        &e (N_{2j}^{-1/2}(a_1' +c_1)x_1+ (a_2' N_{2j}^{-1} +c_2 N_{2j}^{-1/2} )x_2  ) e(-a'\cdot u )dx.
    \end{split}
    \end{equation*}
    If we show that 
    \begin{equation}\label{Tg1,2}
        \norm{Tg}_{L_u^1L_z^2([0,1]^2)} \lesssim_\epsilon N_{2j}^\epsilon \norm{g}_{L^1(\bR^2)}  \qquad \forall g \in L^1(\bR^2)
    \end{equation}
    and
    \begin{equation}\label{Tg2,2}
        \norm{Tg}_{L_u^2L_z^2([0,1]^2)} \lesssim N_{2j}^{3/4} \norm{g}_{L^2(\bR^2)} \qquad \qquad \forall g \in L^2(\bR^2),
    \end{equation}
    then the mixed norm interpolation theorem (see \cite[Theorem 2 in Section 7]{BP61}) implies that
    \begin{equation}\label{estG_c}
        \norm{G_c(u,z)}_{L_u^q L_z^2([0,1]^2)} = \norm{Tg}_{L_u^q L_z^2([0,1]^2)} \lesssim_\epsilon N_{2j}^{\epsilon+\frac{3}{2}\left(1-\frac{1}{q}\right)} \norm{g}_{L^q(\bR^2)} \qquad \forall g \in L^q(\bR^2).
    \end{equation}
    
    When $q=1$, we only need to consider the case $c=0$, since the similar argument works for $c \neq 0$. We have
    \begin{equation*}
            Tg(u,z) =\sum_{|a_1'| \lesssim N_{2j}^{1/2} , |a_2'| \lesssim N_{2j}} \int_{\bR^2} g(N_{2j}^{1/2}x_1, N_{2j}x_2) e(x_1z_1 + N_{2j}^{1/2} x_2 z_2) N_{2j}^{3/2} e(-a'\cdot (u-x) )dx.
    \end{equation*}
    Note that $\sum_{|a_1'| \lesssim N_{2j}^{1/2} , |a_2'| \lesssim N_{2j}}e(-a'\cdot u)$ is a product of Dirichlet kernels, since 
    \begin{equation*}
        \sum_{|a_1'| \lesssim N_{2j}^{1/2} , |a_2'| \lesssim N_{2j}}e(-a' \cdot u) = \sum_{|a_1'| \lesssim N_{2j}^{1/2}} e(-a_1' u_1) \sum_{|a_2'| \lesssim N_{2j}} e(-a_2'u_2). 
    \end{equation*}
    Uniformly in $x$, we have
    \begin{equation*}
        \norm{\sum_{|a_1'| \lesssim N_{2j}^{1/2} , |a_2'| \lesssim N_{2j}}e(-a' \cdot (u-x))}_{L^1([0,1]^2,du)} \lesssim \log(N_{2j}^{1/2}) \log (N_{2j}).
    \end{equation*}
    Therefore, we obtain that
    \begin{equation}\label{L1du}
        \begin{split}
            \norm{Tg}_{L^1([0,1]^2,du)} &\lesssim   \int_{\bR^2} \int_{[0,1]^2} \bigg| g(N_{2j}^{1/2} x_1 , N_{2j} x_2 ) N_{2j}^{3/2}\bigg|\bigg| \sum_{|a_1'| \lesssim N_{2j}^{1/2}, |a_2'| \lesssim N_{2j} } e(-a' \cdot (x-u))\bigg| dudx  \\
            &\lesssim \log(N_{2j})^2 \int_{\bR^2 }  \bigg| g(N_{2j}^{1/2} x_1 , N_{2j} x_2 ) N_{2j}^{3/2}\bigg| dx\\
            &\lesssim_{\epsilon} N_{2j}^\epsilon \norm{g}_{L^1(\bR^2)}.
        \end{split}
    \end{equation}
    Since the estimate \eqref{L1du} holds uniformly in $z$, we established \eqref{Tg1,2}.
    
    When $q=2$, since $Tg(u,z)$ is a Fourier series with respect to $u$ variable, we use Plancherel's theorem and by \eqref{chvar}, we have
    \begin{equation*}
    \begin{split}
        &\norm{Tg}_{L_u^2L_z^2([0,1]^2)} \\
        &= \left( \int_{[0,1]^2} \sum_{|a_1'| \lesssim N_{2j}^{1/2} , |a_2'| \lesssim N_{2j}} |\widehat{g} (N_{2j}^{-1/2} (z_1 +a_1'+c_1), N_{2j}^{-1/2} (z_2+c_2) + a_2' N_{2j}^{-1} )|^2  dz \right)^{1/2}\\
        &=N_{2j}^{1/2} \left( \int |\widehat{g}(\xi)|^2 \sum_{|a_1'| \lesssim N_{2j}^{1/2} , |a_2'| \lesssim N_{2j}} \mathbf{1}_{Q_{a'c}}(\xi) d\xi  \right)^{1/2}
    \end{split}
    \end{equation*}
    where
    \begin{equation*}
        \mathbf{1}_{Q_{a',c}}(\xi)=\mathbf{1}_{[0,1]^2} (N_{2j}^{1/2} \xi_1 - (a_1'+c_1), N_{2j}^{1/2} \xi_2 -(a_2'N_{2j}^{-1/2}+c_2)  ).
    \end{equation*}
    The function $\mathbf{1}_{Q_{a',c}}(\xi)$ is supported on a square with side length $O(N_{2j}^{-1/2})$ centered at $( (a_1'+c_1)N_{2j}^{-1/2},  a_2'N_{2j}^{-1}+c_2N_{2j}^{-1/2}) $. These squares overlap at most $O(N_{2j}^{1/2})$ times. Therefore, we have
    \begin{equation*}
        \norm{Tg}_{L_u^2L_z^2([0,1]^2)} \lesssim N_{2j}^{3/4} \norm{\widehat{g}}_{L^2(\bR^2)} \leq N_{2j}^{3/4} \norm{g}_{L^2(\bR^2)}.
    \end{equation*}
   Therefore, we established \eqref{Tg2,2}. Combining \eqref{spltF_cG_c}, \eqref{estF_c} and \eqref{estG_c}, we obtain that
   \begin{equation} \label{LReqlast-1}
       \norm{\sum_{a \in A_{2j}} f_{\Omega_a} }_{L^p(B^2(N_j))} \lesssim_\epsilon K_p(P_{2j}) N_{2j}^{\epsilon + \frac{3}{2p}-\frac{1}{2}} \left( \int \sum_{a \in A_{2j}} |\widehat{f}_{\Omega_a} \ast \widehat{\eta}_{2j}(\xi)|^2 d\xi  \right)^{1/2}. 
   \end{equation}
   
   We define
   \begin{equation*}
       \mathbf{1}_{\ell,2j,a}(\xi) := \sum_{b \in A_{\ell,2j,a}} \mathbf{1}_{[0,1]\times [-1,1]} (N_{\ell}^{1/2}\xi_1-b , N_{\ell}\xi_2 -(2b+1) N_{\ell}^{1/2}\xi_1 + b^2+b)
   \end{equation*}
   so that $\widehat{f}_{\Omega_a} (\xi) = \widehat{f}_{\Omega_a} \mathbf{1}_{\ell,2j,a}(\xi)$. By H\"{o}lder's inequality, we have 
   \begin{equation*}
   \begin{split}
        |\widehat{f}_{\Omega_a} \ast \widehat{\eta}_{2j}(\xi)|^2  &\leq  (|\widehat{f}_{\Omega_a} \mathbf{1}_{\ell,2j,a}| \ast |\widehat{\eta}_{2j}|(\xi) )^2\\
       &\leq (|\widehat{f}_{\Omega_a}|^2 \ast |\widehat{\eta}_{2j}| (\xi)) ( \mathbf{1}_{\ell,2j,a} \ast |\widehat{\eta}_{2j}|(\xi)).
   \end{split}
   \end{equation*}
   Since $|\widehat{\eta}_{2j}|\lesssim N_{2j}$, for any fixed $\xi$, we have
   \begin{equation*}
       (\mathbf{1}_{\ell,2j,a} \ast |\widehat{\eta}_{2j}| (\xi)) \leq \norm{\widehat{\eta}_{2j}}_{L^\infty(\bR^2)} \norm{\mathbf{1}_{\ell,2j,a}}_{L^1(\bR^2)} \lesssim N_{2j}N_{\ell}^{-3/2} \frac{|A_\ell|}{|A_{2j}|} .
   \end{equation*}
   Since $N_{2j}^{\alpha/2} \leq |A_{2j}|$ and $\norm{\widehat{\eta}_{2j}}_{L^1(\bR^2)} \lesssim 1$, we get
   \begin{equation}\label{LReqlast}
       \begin{split}
           \int \sum_{a \in A_{2j}} |\widehat{f}_{\Omega_a} \ast \widehat{\eta}_{2j}(\xi)|^2 d\xi  &\lesssim N_{2j}^{1-\alpha/2} N_\ell^{-3/2} |A_\ell| \int \sum_{a \in A_{2j}} |\widehat{f}_{\Omega_a}|^2 \ast |\widehat{\eta}_{2j}|(\xi) d\xi \\
           &\lesssim N_{2j}^{1-\alpha/2} N_\ell^{-3/2} |A_\ell| \sum_{a \in A_{2j} } \norm{ \widehat{f}_{\Omega_a}}_{L^2(\bR^2)}^2\\
           &\lesssim N_{2j}^{1-\alpha/2} N_\ell^{-3/2}|A_\ell| \norm{\sum_{a \in A_{2j} } \widehat{f}_{\Omega_a}}_{L^2(\bR^2)}^2.
       \end{split}
   \end{equation}
   By \eqref{LReqlast-1} and \eqref{LReqlast}, we obtain \eqref{eqLemLR}.
\end{proof}

\section{Fourier decay}\label{secFD}
We abbreviate almost surely to a.s. When we say an inequality holds a.s., the corresponding implicit constant may depend on the measure $\nu$, but the probability that the constant exist is $1$. Let us consider a nondecreasing sequence $\{m_j\}_{ j \in \bN }$ such that $2 \leq m_j$. Let $M_j := m_{j} \cdots m_1$ and assume that for any $ \epsilon >0$ and $ j \in \bN$,
\begin{equation}\label{FDcond1}
    m_{j+1} \lesssim_\epsilon M_j^\epsilon.
\end{equation}
Let $\cI_j$ be the collection of $M_j^{-1}$-intervals such that
\begin{equation*}
    \cI_j =\{ M_j^{-1} (k + [0,1]) : 0 \leq k \leq M_j-1 \}.
\end{equation*}
We consider a sequence of random functions $\mu_j $ which satisfies the following conditions for some deterministic nondecreasing sequence $\{ \beta_j\}_{j \in \bN}$:
\begin{itemize}
    \item $\mu_0 = \mathbf{1}_{[0,1]}$.
    \item $\mu_j = \beta_j \mathbf{1}_{E_j}$ where $E_j$ is a union of intervals in $\cI_j$.
    \item $\mathbb{E}(\mu_{j+1}(x)| E_j) = \mu_j(x)$ for all $x \in [0,1]$.
    \item Conditioned on $E_j$, the sets $I_j \cap E_{j+1}$ are chosen independently for each $I_j \subseteq E_j$.
\end{itemize}
We will identify the functions $\mu_j$ with the measures $\mu_j dx$ and write $\norm{\mu_j}:=\mu_j([0,1])$.
\begin{remark}
    First three conditions are same with (M1)-(M3) in \cite{ShSu17} and the last condition is weaker than (M4) in \cite{ShSu17}. Since we only prove Fourier decay, it suffices to assume the weaker condition.
\end{remark}

Let $\alpha_0$ be a number such that
\begin{equation}\label{FDcond2}
    1-\alpha_0 = \lim_{j \rightarrow \infty} \frac{\log \beta_j}{\log M_j}.
\end{equation}
Note that \eqref{FDcond2} implies that there exist $j_0(\epsilon)$ for every $\epsilon>0$ such that if $j \geq j_0(\epsilon)$, then
\begin{equation}\label{FDcond2_1}
	M_j^{1-\alpha_0 -\epsilon} \leq \beta_j \leq M_j^{1-\alpha_0+\epsilon}.
\end{equation}

Shmerkin and Suomala showed that $\mu_j$ converges weakly to a measure $\mu$ supported on $[0,1]$ and the support of the measure ${\mu}$  is a Salem set a.s. so that $|\widehat{\mu}(\xi)| \lesssim_\sigma (1+|\xi|)^{-\sigma/2} $ for any $\sigma < \alpha_0$, see \cite[Theorem 4.2]{ShSu17} and \cite[Theorem 14.1]{ShSu18}.

Now, let $\nu_j$ be a measure defined as
\begin{equation*}
    \int f(x_1,x_2) d\nu_j := \int f(x_1,x_1^2) \mu_j(x_1) dx_1. 
\end{equation*}
Similarly, $\nu_j$ converges weakly to a measure $\nu$ supported on $\bP^1$.
\begin{prop}\label{PropFDecay}
    Suppose that $\mu_j$ is a sequence of random measures that satisfies the conditions above. For any $0< \sigma <\alpha_0$, the limit measure $\nu$ satisfies the following inequality a.s.
    \begin{equation}\label{Fdecay}
        |\widehat{\nu}(\xi)| \lesssim_{\sigma} (1+|\xi|)^{-\sigma/2} \qquad \forall \xi \in \bR^2.
    \end{equation}
\end{prop}
\begin{remark}
    For each $0< \beta <1$, Kaufman \cite{Kauf76} constructed a Salem set of dimension $\beta$ which is a subset of $\bP^1$. The proof of Proposition \ref{PropFDecay} is similar to \cite{Kauf76} in that we used probabilistic argument and Van der Corput lemma. However, we showed that \eqref{Fdecay} is satisfied a.s. in more general settings.
\end{remark}
One of the main ingredients of the proof is Hoeffding's inequality.
\begin{thm}[Hoeffding's inequality \cite{Ho63}]
    Let $X_1, \cdots, X_n$ be independent real random variables such that $a_i \leq X_i \leq b_i$ and $S_n := X_1 + \cdots + X_n$. For $t>0$,
    \begin{equation}\label{eqhoeff}
        \bP \left( \bigg|S_n - \mathbb{E}(S_n) \bigg|>t \right) \leq 2 \exp{\left(\frac{-2t^2}{\sum_{i=1}^n (b_i-a_i)^2} \right)}.
    \end{equation}
\end{thm}
According to the size of $|\xi|$, we will use different arguments to obtain estimates for $|\widehat{\nu}_{j+1}(\xi) - \widehat{\nu}_j(\xi)|$. If $|\xi| \leq M_{j+1}$, we use \eqref{eqhoeff}. If $|\xi| \geq m_jM_{j+1}^2$, we use van der Corput lemma. Lastly, if $M_{j+1} \leq |\xi| \leq m_{j+1}M_{j}^2$, we use both of them. 
\begin{lem}\label{FDlemSmall}
    For fixed $\xi$ such that $|\xi| \leq M_{j+1}$, we have the following tail bound. There exists a constant $C_1>0 $ such that
    \begin{equation}\label{Fdlemsmalleq}
        \bP\left( |\widehat{\nu}_{j+1} (\xi) -\widehat{\nu}_j(\xi) | \geq M_j^{-\sigma/2} \norm{\mu_j}^{1/2} \bigg| E_j \right)  \lesssim \exp(-C_1 M_{j}^{1-\sigma} \beta_j \beta_{j+1}^{-2}).
    \end{equation}
\end{lem}
\begin{proof}
    The proof is similar to the proof of Theorem 4.2 in \cite{ShSu17}. For $I_j \in \cI_j$, we consider
    \begin{equation*}
        X_{I_j} = \int_{I_j} \beta_{j+1} \mathbf{1}_{E_{j+1}} e(-x \xi_1 - x^2 \xi_2) dx
    \end{equation*}
    for $\xi=(\xi_1,\xi_2) \in \bR^2$. Then, $\sum_{I_j \in \cI_j} X_{I_j} = \widehat{\nu}_{j+1}(\xi)$ and $\mathbb{E}(\mu_{j+1}(x)| E_j) = \mu_j(x)$ implies that
    \begin{equation*}
        \begin{split}
            \mathbb{E}(\sum_{I_j\in \cI_j} X_{I_j} |E_j) &= \mathbb{E}(\widehat{\nu}_{j+1}(\xi) | E_j)\\
            &= \int e(-x\xi_1 - x^2\xi_2 ) \mathbb{E}(\mu_{j+1}(x)|E_j)dx\\
            &=\int e(-x\xi_1 -x^2\xi_2) \mu_j(x)dx = \widehat{\nu}_j(\xi).
        \end{split}
    \end{equation*}
    Let $S_{j} : = \sum_{I_j \in \cI_j} X_{I_j}$, then 
    \begin{equation*}
        |S_{j} - \mathbb{E}(S_{j})| = |\widehat{\nu}_{j+1}(\xi) - \widehat{\nu}_j(\xi)|
    \end{equation*}
    conditioned on $E_j$. Also, note that $X_{I_j}$ are independent conditioned on $E_j$, since $I_j \cap E_{j+1}$ are chosen independently.
    
    Next, we need to estimate $ |X_{I_j}|$ for $I_j \subseteq E_j$. The random variables $X_{I_j}$ are not real-valued. But by considering real and imaginary parts of $X_{I_j}$ separately, we can use Hoeffding's inequality \eqref{eqhoeff} with $\sum_{i=1}^n (b_i-a_i)^2$ replaced by the upper bound of $\sum_{I_j }|X_{I_j}|^2$.
    
    We have
    \begin{equation}\label{FDtrivial}
        |X_{I_j}| \leq \beta_{j+1} M_j^{-1}
    \end{equation}
    and the number of $I_j \subseteq E_j$ is $M_j\beta_j^{-1} \norm{\mu_j}$. Thus, we obtain
    \begin{equation*}
        \sum_{I_j \subseteq E_j} |X_{I_j}|^2 \leq\frac{\beta_{j+1}^2}{\beta_j M_j} \norm{\mu_j}.
    \end{equation*}
    For $t = M_{j}^{-\sigma/2} \norm{\mu_j}^{1/2}$, Hoeffding's inequality \eqref{eqhoeff} implies \eqref{Fdlemsmalleq}.
    \end{proof}
    Then, the unconditional version of Lemma \ref{FDlemSmall} easily follows.
    \begin{cor}\label{FDlemSmall2}
    For fixed $\xi$, if $|\xi| \leq M_{j+1}$, then there exists a constant $C_1>0 $ such that
    \begin{equation*}
        \bP\left( |\widehat{\nu}_{j+1} (\xi) -\widehat{\nu}_j(\xi) | \geq M_j^{-\sigma/2} \norm{\mu_j}^{1/2} \right)  \lesssim \exp(-C_1 M_{j}^{1-\sigma} \beta_j \beta_{j+1}^{-2}).
    \end{equation*}
    \end{cor}
    \begin{proof}
        Note that Lemma \ref{FDlemSmall} holds uniformly in $E_j$. Thus, the result can be immediately derived from the law of total probability.
    \end{proof}
    For the rest of this section, we will use the following notations additionally. For fixed $\xi=(\xi_1,\xi_2) \in \bR^2$ such that $\xi \neq 0$, consider $I_{j+1}(k) := M_{j+1}^{-1}(k+[0,1]) \in \mathcal{I}_{j+1} $ and let $m(k)$ be the minimum of $|\xi_1 + 2x\xi_2|$ on $ I_{j+1}(k)$. Also, if $\xi_2 \neq 0$, let $m(k_0)$ be the smallest among $m(k)$ and $x_0$ be a point in $[0,1]$ such that $m(k_0)= |\xi_1+2x_0\xi_2|$. If there are two $k$'s such that $m(k)$ is the smallest, $k_0$ can be either of them.
    \begin{lem}\label{FDlemLarge}
    For any $|\xi| \geq m_{j+1}M_{j}^2$, we have
    \begin{equation}\label{FDLarge}
        |\widehat{\nu}_{j+1}(\xi) -\widehat{\nu}_j(\xi) | \lesssim_\epsilon M_{j+1}^{1-\alpha_0+ \epsilon}  |\xi|^{-1/2}.
    \end{equation}
    \end{lem}
    \begin{proof}
    By Van der Corput lemma (see for example \cite[Proposition 2.2, Chapter 8]{StSh11}), we get
    \begin{equation}\label{mkneq0}
        \bigg| \int_{I_{j+1}(k)}  e(-x\xi_1 -x^2\xi_2)dx \bigg| \lesssim m(k)^{-1} \qquad \mathrm{if} \ m(k)\neq 0.
    \end{equation}
    If $|\xi_1 | \geq 10 |\xi_2|$, then $|\xi_1 + 2x\xi_2| \gtrsim |\xi|$. Therefore,
    \begin{equation}\label{FDlargebigxi_1}
        \begin{split}
            |\widehat{\nu_{j+1}}(\xi)| &\leq \sum_{I_{j+1}(k) \subseteq E_{j+1}} \bigg| \int_{I_{j+1}(k)} \beta_{j+1} e(-x\xi_1 -x^2\xi_2) dx \bigg| \\
            &\lesssim \beta_{j+1}M_{j+1}|\xi|^{-1}\\
            &\leq \beta_{j+1} m_{j+1}^{1/2} |\xi|^{-1/2}.
        \end{split}
    \end{equation}
    In the last inequality, we used that $|\xi| \geq m_{j+1}M_j^2$.
    
    If $|\xi_1 | \leq 10 |\xi_2|$, we consider $I_{j+1}(k_0+p)$ for $p =-1,0$ or $1$. The interval $I_{j+1}(k_0+p) $ can be divided into $ I_1 $ and $ I_2$ where $I_1 = \{x : |x-x_0| \leq |\xi|^{-1/2} \}$ and $I_2 = \{ x : |x-x_0| \geq |\xi|^{-1/2} \}$. For $x \in I_2$, since $|\xi_1 + 2x\xi_2|$ attains its minimum at $x_0$,
    \begin{equation*}
    \begin{split}
        |\xi_1 +2x\xi_2 | &= |(\xi_1 +2x_0 \xi_2) +2(x-x_0)\xi|\\
        &\geq 2|x-x_0| |\xi| \gtrsim |\xi|^{1/2}.
    \end{split}
    \end{equation*}
    Therefore, \eqref{mkneq0} implies that
    \begin{equation}\label{sec5intk_0}
    \begin{split}
        \bigg| \int_{I_{j+1}(k_0+p)} \beta_{j+1} e(-x\xi_1-x^2\xi_2) dx \bigg| &\lesssim \beta_{j+1} |I_1| + \beta_{j+1}|\xi|^{-1/2}\\
        &\lesssim \beta_{j+1} |\xi|^{-1/2}.
    \end{split}
    \end{equation}
    If $p \neq -1,0,1$, then $m(k_0+p) \geq 2|x-x_0||\xi| \gtrsim p M_{j+1}^{-1}|\xi| $. Using \eqref{mkneq0}, we obtain
    \begin{equation}\label{sec5intk_0+p}
        \bigg| \int_{I_{j+1}(k_0+p)} \beta_{j+1} e(-x\xi_1-x^2\xi_2) dx \bigg| \lesssim \frac{\beta_{j+1}M_{j+1}}{p|\xi|}.
    \end{equation}
    Combining \eqref{sec5intk_0} and \eqref{sec5intk_0+p}, we get
    \begin{equation*}
        \begin{split}
            |\widehat{\nu_{j+1}}(\xi)| &\leq \sum_{I_{j+1}(k) \subseteq E_{j+1}} \bigg| \int_{I_{j+1}(k)} \beta_{j+1} e(-x\xi_1 -x^2\xi_2) dx \bigg| \\
            &\lesssim \beta_{j+1}|\xi|^{-1/2} + \sum_{p=1}^{M_{j+1}} \frac{\beta_{j+1}M_{j+1}}{p|\xi|}\\
            &\lesssim \beta_{j+1}m_{j+1}^{1/2} \log(M_{j+1}) |\xi|^{-1/2}.
        \end{split}
    \end{equation*}
    
    We considered when $|\xi_1| \geq 10|\xi_2|$ and $|\xi_1 | \leq 10|\xi_2|$. In either cases, \eqref{FDcond1} and \eqref{FDcond2_1} imply that
    \begin{equation*}
    \begin{split}
        |\widehat{\nu_{j+1}}(\xi)| &\lesssim \beta_{j+1}m_{j+1}^{1/2} \log(M_{j+1}) |\xi|^{-1/2}\\
            &\lesssim_\epsilon M_{j+1}^{1-\alpha_0+ \epsilon} |\xi|^{-1/2}
    \end{split} 
    \end{equation*}
    for any $\xi$. The same argument works for $\widehat{\nu}_{j}$ with $j+1$ replaced by $j$. Thus, we have $|\widehat{\nu}_{j}(\xi)| \lesssim_\epsilon M_j^{1-\alpha_0+\epsilon} |\xi|^{-1/2}$. Since $M_j \leq M_{j+1}$, we get \eqref{FDLarge}.
\end{proof}

\begin{lem}\label{FDlemmid}
    For fixed $\xi$, if $M_{j+1} \leq |\xi| \leq m_{j+1} M_j^2$, let $k$ be an integer such that $0 \leq k \leq j-1$ and
    \begin{equation*}
        M_{j+1}M_k\leq |\xi| \leq M_{j+1}M_{k+1}.
    \end{equation*}
    For such $k$, the following tail bounds hold. \\
    If $|\xi_1| \geq 10|\xi_2| $, there exists a constant $C_2>0 $ such that
    \begin{equation}\label{FDlemmideq1}
        \bP\left( |\widehat{\nu}_{j+1} (\xi) -\widehat{\nu}_j(\xi) | \geq M_{j}^{-(\alpha_0-\epsilon)/2} \left(  \frac{M_j}{|\xi|} \right)^{1/2}  \norm{\mu_j}^{1/2} \bigg| E_j \right)  \lesssim \exp(-C_2 M_{j}^{1-\alpha_0+\epsilon} \beta_j \beta_{j+1}^{-2} m_{j+1}^{-1})
    \end{equation}
    If $|\xi_1| \leq 10|\xi_2| $, there exists a constant $C_3>0 $ such that
    \begin{equation}\label{FDlemmideq2}
    \begin{split}
        \bP\left( |\widehat{\nu}_{j+1} (\xi) -\widehat{\nu}_j(\xi) | \geq M_{j}^{-(\alpha_0-\epsilon)/2} \bigg[ \sum_{i=0}^{k} \left(\frac{M_{i+1}}{M_k} \right)^2 \mu_j(I_{i}(x_0))\bigg]^{1/2}  \bigg| E_j \right)& \\
        \lesssim \exp(-C_3 M_{j}^{1-\alpha_0+\epsilon} \beta_j \beta_{j+1}^{-2})&
    \end{split}
    \end{equation}
    where $I_i(x_0)$ is an interval of length $M_i^{-1}$ centered at $x_0$.
\end{lem}
\begin{proof}
If $|\xi_1| \geq 10|\xi_2|$, by the argument similar to \eqref{FDlargebigxi_1}, we have $|X_{I_j} | \lesssim \beta_{j+1} m_{j+1}|\xi|^{-1} $ and the number of interval $I_{j}$ in $E_j$ is $\beta_j^{-1} M_j \norm{\mu_j}$. Since $M_{j+1 } \leq |\xi|$, we have
\begin{equation*}
    \begin{split}
        \sum_{I_j \subseteq E_j} |X_{I_j}|^2 &\lesssim \beta_{j+1}^2 \beta_j^{-1} m_{j+1} M_{j+1} |\xi|^{-2} \norm{\mu_j}\\
        & \leq \beta_{j+1}^2 \beta_j^{-1} m_{j+1} |\xi|^{-1} \norm{\mu_j}.
    \end{split}
\end{equation*}
For $t= M_j^{(1-\alpha_0+\epsilon)/2} |\xi|^{-1/2} \norm{\mu_j}^{1/2}$, Hoeffding's inequality \eqref{eqhoeff} implies \eqref{FDlemmideq1}.

If $|\xi_1 | \leq 10|\xi_2|$, we consider $I_k(x_0)$ first. If $I_j \subseteq I_k(x_0)$, we use \eqref{FDtrivial}. Since the number of $I_j$ in $I_k(x_0)$ is $\beta_j^{-1}M_j\mu_j (I_k(x_0))$, we obtain that
\begin{equation}\label{sec5X_Iink}
    \sum_{I_j \subseteq I_k(x_0)} |X_{I_j}|^2 \leq \beta_{j+1}^2 \beta_j^{-1} M_j^{-1} \mu_j(I_k(x_0)).
\end{equation}
If $I\subseteq I_{i}(x_0)$ but $I \not\subseteq I_{i+1}(x_0)$ for $0 \leq i \leq k-1$, then $|x-x_0| \geq M_{i+1}^{-1} - M_j^{-1} \gtrsim M_{i+1}^{-1}$ for any $x \in I$. Therefore, 
\begin{equation*}
    |\xi_1+2x\xi_2| \gtrsim |x-x_0| |\xi| \gtrsim M_{i+1}^{-1} |\xi|.
\end{equation*}
We use \eqref{mkneq0} again and obtain that $|X_{I_j}| \lesssim \beta_{j+1} m_{j+1}M_{i+1}|\xi|^{-1}$.
Since the number of $I_j$ in $I_i(x_0)$ is $\lesssim \beta_j^{-1}M_j \mu_j(I_i(x_0))$ and $M_{j+1}|\xi|^{-1} \leq M_k^{-1}$, we get
\begin{equation}
\begin{split}\label{sec5X_Iini}
    \sum_{I_j \subseteq I_i(x_0)} |X_{I_j}|^2 &\lesssim \frac{\beta_{j+1}^2}{\beta_{j}M_j} \frac{M_{j+1}^2}{|\xi|^2} M_{i+1}^2 \mu_j(I_i(x_0))\\
    &\leq \frac{\beta_{j+1}^2}{\beta_j M_j} \left(\frac{M_{i+1}}{M_k} \right)^2 \mu_j(I_i(x_0)).
\end{split} 
\end{equation}
Combining \eqref{sec5X_Iink} and \eqref{sec5X_Iini}, we get
\begin{equation*}
    \sum_{I_j \subseteq E_j} |X_{I_j}|^2 \lesssim \frac{\beta_{j+1}^2}{\beta_j M_j} \sum_{i=0}^k \left(\frac{M_{i+1}}{M_k} \right)^2 \mu_j(I_i(x_0)).
\end{equation*}
For
    \begin{equation*}
        t = M_{j}^{-(\alpha_0-\epsilon)/2} \left( \sum_{i=0}^{k} \left(\frac{M_{i+1}}{M_k} \right)^2 \mu_j(I_i(x_0))\right)^{1/2},
    \end{equation*}
    Hoeffding's inequality \eqref{eqhoeff} implies \eqref{FDlemmideq2}.
\end{proof}
We also have the following unconditional version of Lemma \ref{FDlemmid}.
    \begin{cor}\label{FDlemmid2}
    For fixed $\xi$, if $M_{j+1} \leq |\xi| \leq m_{j+1} M_j^2$, for $k$ defined in Lemma \ref{FDlemmid}, we have the following tail bounds. \\
    If $|\xi_1| \geq 10|\xi_2| $, there exists a constant $C_2>0 $ such that
    \begin{equation*}
        \bP\left( |\widehat{\nu}_{j+1} (\xi) -\widehat{\nu}_j(\xi) | \geq M_{j}^{-(\alpha_0-\epsilon)/2} \left(  \frac{M_j}{|\xi|} \right)^{1/2}  \norm{\mu_j}^{1/2} \right)  \lesssim \exp(-C_2 M_{j}^{1-\alpha_0+\epsilon} \beta_j \beta_{j+1}^{-2} m_{j+1}^{-1})
    \end{equation*}
    If $|\xi_1| \leq 10|\xi_2| $, there exists a constant $C_3>0 $ such that
    \begin{equation*}
    \begin{split}
        \bP\left( |\widehat{\nu}_{j+1} (\xi) -\widehat{\nu}_j(\xi) | \geq M_{j}^{-(\alpha_0-\epsilon)/2} \bigg[ \sum_{i=0}^{k} \left(\frac{M_{i+1}}{M_k} \right)^2 \mu_j(I_{i}(x_0))\bigg]^{1/2}  \right)& \\
        \lesssim \exp(-C_3 M_{j}^{1-\alpha_0+\epsilon} \beta_j \beta_{j+1}^{-2})&
    \end{split}
    \end{equation*}
    where $I_i(x_0)$ is an interval of length $M_i^{-1}$ centered at $x_0$.
    \end{cor}
    \begin{proof}
        Note that Lemma \ref{FDlemmid} holds uniformly in $E_j$. Thus, the result immediately follows from the law of total probability.
    \end{proof}

In order to prove Proposition \ref{PropFDecay}, we also need an estimate for $\mu_j(I_i(x_0))$.
\begin{lem}\label{sec5lemmuI_i}
    We have the following estimate a.s. 
    \begin{equation*}
     \mu_j(I_i(x_0)) \lesssim_\epsilon M_j^{\epsilon} \beta_i M_i^{-1}   
    \end{equation*}
    where $j \geq i$ and the implicit constant is uniform in $i, j$ and $x_0$.
\end{lem}
\begin{proof}
    Note that $I_i(x_0) \subseteq I_i(k-1) \cup I_i(k) \cup I_i(k+1)$ for some $0 \leq k \leq M_i-1$. Thus, it suffices to show that for any $j \geq i$,
    \begin{equation}\label{sec5lemmuI_ieq}
     \mu_j(I_i) \lesssim_\epsilon M_j^{\epsilon} \beta_i M_i^{-1} \qquad \text{a.s.}
    \end{equation}
    where the implicit constant is uniform in $i, j$ and $I_i \in \mathcal{I}_i$.
    
    For fixed $I_i \in \mathcal{I}_i$, let $X_{I_{j}} = \mu_{j+1}(I_j \cap I_i)$. Then, $\sum_{I_j \in \mathcal{I}_j}X_{I_j} = \mu_{j+1}(I_i)$ and
    \begin{equation*}
        \mathbb{E}(\sum_{I_j \in \mathcal{I}_j} X_{I_j} |E_j) = \sum_{I_j \in \mathcal{I}_j} \mathbb{E} (\mu_{j+1}(I_j \cap I_i) |E_j) = \sum_{I_j \in \mathcal{I}_j} \mu_j(I_j \cap I_i)= \mu_j(I_i).
    \end{equation*}
    Also, $X_{I_j}$ are independent conditioned on $E_j$, since $I_j \cap E_{j+1}$ are chosen independently.
    
    We have $|X_{I_j}| \leq \beta_{j+1}M_j^{-1}$ and the number of $I_j$ such that $I_j \subseteq I_i $ is $\beta_{j}^{-1}M_j \mu_j(I_i)$. Note that $I_j \subseteq I_i$ if and only if $X_{I_j} >0$. Therefore, we have
    \begin{equation*}
        \sum_{I_j \in \mathcal{I}_j} |X_{I_j}|^2 \leq \frac{\beta_{j+1}^2}{\beta_j M_j} \mu_j (I_i).
    \end{equation*}
    For $t= M_j^{\epsilon/4} M_j^{-1/2}\beta_j^{1/2} \mu_j(I_i)^{1/2}$, Hoeffding's inequality \eqref{eqhoeff} implies that there exist a constant $C>0$ such that 
    \begin{equation*}
        \mathbb{P}(|\mu_{j+1}(I_i)  -\mu_j(I_i)| \geq t  |E_j) \lesssim \exp (-C \beta_j^2\beta_{j+1}^{-2} M_j^{\epsilon/2}).
    \end{equation*}
    Therefore, we have
    \begin{equation}\label{sec5mu3I_i}
        \mathbb{P}(|\mu_{j+1}(I_i)  -\mu_j(I_i)| \geq t  \ \mathrm{for \ some \ } I_i \in \mathcal{I}_i ) \lesssim M_i \exp (-C \beta_j^2\beta_{j+1}^{-2} M_j^{\epsilon/2}).
    \end{equation}
    By \eqref{FDcond2_1},
    \begin{equation*}
        \sum_{i=1}^\infty \sum_{j=i}^\infty M_i \exp(-C \beta_j^2\beta_{j+1}^{-2} M_j^{\epsilon/2}) \lesssim_\epsilon \sum_{i=1}^\infty M_i \exp (-CM_i^{\epsilon/4}) < \infty.
    \end{equation*}
    Borel-Cantelli lemma and \eqref{sec5mu3I_i} imply that there exist $i_0$ a.s. such that if $j\geq i \geq i_0$, then for any $I_i \in \cI_i$,
    \begin{equation}\label{sec5itereq}
        \mu_{j+1}(I_i) \leq \mu_j(I_i) + M_j^{\epsilon/4} M_j^{-1/2}\beta_j^{1/2} \mu_j(I_i)^{1/2}.
    \end{equation}
    Using \eqref{FDcond2_1} again, we can choose $i_1(\epsilon) $ such that if $i \geq i_1(\epsilon)$, then 
    \begin{equation*}
        M_i^{1-\alpha_0-\epsilon/4} \leq \beta_i \leq M_i^{1-\alpha_0+\epsilon/4}. 
    \end{equation*}
    
    Now, we claim that if $ i \geq \max(i_0, i_1(\epsilon))$, then there exists a constant $C' \geq 1$ such that for any $j\geq i$ and $I_i \in \cI_i$
    \begin{equation}\label{sec5claim}
        \mu_{j}(I_i) \leq C' j^2 \beta_i M_i^{-1} M_j^{\epsilon(1-2^{-j})}.
    \end{equation}
    If $j =i$, then $\mu_j(I_i) \lesssim \beta_iM_i^{-1}$ uniformly in $i$ and $I_i \in \cI_i$. If \eqref{sec5claim} holds for $j=k$, \eqref{sec5itereq} implies that 
    \begin{equation*}
    \begin{split}
        \mu_{k+1}(I_i) \leq C' k^2 \beta_iM_i^{-1}M_k^{\epsilon(1-2^{-k})} +(C')^{1/2} k\beta_iM_i^{-1} M_k^{\frac{\epsilon}{2}(1-2^{-k})} M_k^{\epsilon/4} \left( \frac{\beta_kM_k^{-1}}{\beta_iM_i^{-1} } \right)^{1/2}.
    \end{split}
    \end{equation*}
    By the choice of $\epsilon_1(\epsilon)$,
    \begin{equation*}
        \frac{ \beta_kM_k^{-1}}{\beta_iM_i^{-1}} \leq \frac{M_k^{-\alpha_0 + \epsilon/4}}{M_i^{-\alpha_0 -\epsilon /4}} \leq (M_kM_i)^{\epsilon/4} \leq M_k^{\epsilon/2}. 
    \end{equation*}
    Thus, we obtain that for any $I_i \in \cI_i$,
    \begin{equation*}
        \mu_{k+1}(I_i) \leq C' (k^2+k) \beta_i M_i^{-1} M_k^{\epsilon(1-2^{-k-1})} \leq C' (k+1)^2 \beta_i M_i^{-1} M_{k+1}^{\epsilon(1-2^{-k-1})}.
    \end{equation*}
    Hence, we proved \eqref{sec5claim}. The estimate \eqref{sec5lemmuI_ieq} follows when $j \geq i \geq \max(i_0,i_1(\epsilon))$. 
    
    To consider the case $i \leq \max(i_0,i_1(\epsilon))$, we first show that $\norm{\mu_j}$ is a.s. bounded. We have $\mathbb{E}(\norm{\mu_{j+1}} |E_j) = \norm{\mu_j}$ for all $E_j$ and the law of total probability implies that $\mathbb{E}(\norm{\mu_{j+1}}) = \mathbb{E}(\norm{\mu_j})$. By iterating the equality, we get $\mathbb{E}(\norm{\mu_j}) = \mathbb{E}(\norm{\mu_0}) =1 $ for all $j \in \bN \cup \{0\}$. By Martingale convergence theorem (see for example \cite[Theorem 4.2.11]{Dur19}), as $j \rightarrow \infty$, $\norm{\mu_j}$ converges to $\norm{\mu}$ a.s. with $\mathbb{E}{(\norm{\mu})} =1$. Therefore, $\norm{\mu_j}$ is a.s. bounded.
    
    Therefore, if $i \leq \max(i_0,i_1(\epsilon))$, for any $j \geq i$ and $I_i \in \cI_i$ we have
    \begin{equation*}
        \mu_j(I_i) \leq \norm{\mu_j} \lesssim_\epsilon \min_{i \leq \max(i_0,i_1(\epsilon))} (\beta_iM_i^{-1}) \leq M_j^\epsilon \beta_i M_i^{-1}  \qquad \mathrm{a.s.}
    \end{equation*}
    Thus, we get \eqref{sec5lemmuI_ieq}.
\end{proof}
\begin{proof}[Proof of Proposition \ref{PropFDecay}]
    As in \cite{ShSu17} and \cite{ShSu18}, we use Lemma 9A4 in \cite{Wo03}. Then, it suffices to consider only $\xi \in \bZ^2$. Let $\Omega_j$ be the event that there exists $\xi \in \bZ^2$ such that one of the following happens.\\
    (1) If $|\xi| \leq M_{j+1}$, then
    \begin{equation*}
        |\widehat{\nu}_{j+1}(\xi) - \widehat{\nu}_j(\xi)| \geq M_j^{-\sigma/2} \norm{\mu_j}^{1/2}. 
    \end{equation*}
    (2) If $M_{j+1} \leq |\xi| \leq m_{j+1}M_{j}^2$ and $|\xi_1 | \geq 10 |\xi_2|$, then
    \begin{equation*}
        |\widehat{\nu}_{j+1}(\xi) - \widehat{\nu}_j(\xi)| \geq M_j^{-(\alpha_0-\epsilon)/2} \left(\frac{M_j}{|\xi|} \right)^{1/2} \norm{\mu_j}^{1/2}. 
    \end{equation*}
    (3) If $M_{j+1} \leq |\xi| \leq m_{j+1}M_{j}^2$ and $|\xi_1 | \leq 10 |\xi_2|$, then
    \begin{equation*}
        |\widehat{\nu}_{j+1} (\xi) -\widehat{\nu}_j(\xi) | \geq  M_j^{-(\alpha_0-\epsilon)/2} \bigg[ \sum_{i=0}^{k} \left( \frac{M_{i+1}}{M_k} \right)^2 \mu_j(I_{i}(x_0)) \bigg]^{1/2} 
    \end{equation*}
    for $k$ such that $0 \leq k \leq j-1$ and $M_{j+1}M_k \leq |\xi| \leq M_{j+1}M_{k+1}$.
    
    Since we only consider $\xi \in \bZ^2$ such that $|\xi| \leq m_{j+1}M_j^2$, Corollary \ref{FDlemSmall2} and \ref{FDlemmid2} imply that
    \begin{equation}\label{POmega}
        \mathbb{P}(\Omega_j) \lesssim m_{j+1}^2 M_j^4 \exp(-C M_j^{1-\alpha_0+\epsilon}\beta_j \beta_{j+1}^{-2}m_{j+1}^{-1})
    \end{equation}
    for some constant $C>0$. For sufficiently large $j$ and sufficiently small $\epsilon >0$, it follows from \eqref{FDcond1} and \eqref{FDcond2_1} that
    \begin{equation*}
        M_{j}^{1-\alpha_0+\epsilon} \beta_j \beta_{j+1}^{-2} m_{j+1}^{-1} \gtrsim_\epsilon M_j^{\epsilon/2}
    \end{equation*}
    Therefore, $\sum_{j=1}^\infty \bP(\Omega_j)  < \infty$.
    
    Using Lemma \ref{sec5lemmuI_i}, \eqref{FDcond2_1} and that $M_{j+1}M_k \leq |\xi| \leq M_{j+1}M_k$, a.s. we have
    \begin{equation}\label{sec5sumi0k}
    \begin{split}
        \sum_{i=0}^{k} \left( \frac{M_{i+1}}{M_k} \right)^2 \mu_j(I_{i}(x_0)) 
        &\lesssim_\epsilon \sum_{i=0}^k M_j^\epsilon \beta_i M_{i}m_{i+1}^2 M_k^{-2}\\
        &\lesssim_\epsilon M_j^\epsilon \beta_k m_{k+1}^2 M_k^{-1} \lesssim_\epsilon M_j^\epsilon M_k^\epsilon \left(\frac{M_{j+1}}{|\xi|} \right)^{\alpha_0-\epsilon}\\
        &\lesssim M_j^{2\epsilon}\left(\frac{M_{j}}{|\xi|} \right)^{\alpha_0-\epsilon} .
    \end{split}
    \end{equation}
    We apply Borel-Cantelli lemma to $\Omega_j$, Since $\norm{\mu_j}$ is a.s. bounded, Lemma \ref{FDlemLarge} and \eqref{sec5sumi0k} imply that
    a.s. there exists $j_0$ such that for $j \geq j_0$,
    \begin{align*}
        |\widehat{\nu}_{j+1}(\xi) -\widehat{\nu}_j(\xi) | &\lesssim  M_{j}^{-\sigma/2} & &\mathrm{if} \ |\xi| \leq M_{j+1},\\
        &\lesssim_\epsilon  M_j^{2\epsilon} |\xi|^{(-\alpha_0+\epsilon)/2}  & &\mathrm{if} \ M_{j+1}\leq  |\xi| \leq m_{j+1}M_{j}^2,\\
        &\lesssim_\epsilon  M_{j+1}^{1-\alpha_0+\epsilon}|\xi|^{-1/2} & &\mathrm{if} \ |\xi|\geq m_{j+1}M_{j}^2.
    \end{align*}
    When $M_{j+1}\leq  |\xi| \leq m_{j+1}M_{j}^2$ and $|\xi_1| \geq 10 |\xi_2|$, we used that $M_j/|\xi| \leq (M_j/|\xi|)^{\alpha_0-\epsilon}$ for $\epsilon < \alpha_0$.
    
    Let $j_1, j_2 $ be numbers such that $M_{j_2} \leq |\xi| \leq M_{j_2+1}$ and $m_{j_1+1}M_{j_1}^2 \leq |\xi| \leq m_{j_1+2}M_{j_1+1}^2$. By telescoping, we obtain for $k \geq j_0$,
    \begin{equation*}
    \begin{split}
        |\widehat{\nu}_k(\xi) - \widehat{\nu}_{j_0}(\xi) | &\lesssim_\epsilon \sum_{j_0 \leq j \leq j_1} M_{j+1}^{1-\alpha_0+\epsilon} |\xi|^{-1/2} +\sum_{\max(j_0,j_1) < j < j_2}  M_{j}^{2\epsilon} |\xi|^{(-\alpha_0+\epsilon)/2} \\
        &\qquad + \sum_{\max(j_0,j_2) \leq j \leq k} M_{j}^{-\sigma/2}.
    \end{split}
    \end{equation*}
    We drop the first sum if $j_1 < j_0$ and we drop the first and the second sum if $j_1,j_2 \leq j_0$.\\
    Since $\sigma< \alpha_0$, for sufficiently small $\epsilon >0$, we get
    \begin{equation*}
        \begin{split}
            |\widehat{\nu}_k(\xi) - \widehat{\nu}_{j_0}(\xi) | &\lesssim_\epsilon M_{j_1+1}^{ 1-\alpha_0+\epsilon}|\xi|^{-1/2} + M_{j_2-1}^{2\epsilon} |\xi|^{(-\alpha_0+\epsilon)/2} + M_{j_2}^{-\sigma/2}\\
            &\lesssim_\epsilon |\xi|^{-\sigma/2+\epsilon}.
        \end{split}
    \end{equation*}
    Since $|\widehat{\nu}_{j_0}(\xi)| \lesssim_{j_0} |\xi|^{-1/2} $, by letting $k \rightarrow \infty$, we obtain \eqref{Fdecay}.
\end{proof}

\section{Global restriction estimate}\label{secGR}
We need to randomize the choice of $S_{j,a}$ in order to get the global restriction estimate.
\begin{prop}\label{propGR}
    Let $\{n_j\}_{j \in \bN}$ be a sequence of positive numbers which satisfies \eqref{nvcond1}, \eqref{nvcond2}, \eqref{nvcond3} and \eqref{nvcond4}. Let $\{\mu_j\}_{j \in \bN \cup \{ 0\} }$ be a sequence of random measures on $[0,1]$ which satisfies the following.\\
    (1) $\mu_0, \mu_1, \cdots$ are constructed through the process in Section \ref{secCan}.\\
    (2) For each $j \in \bN$ and for each $a \in A_j$, the set $S_{j,a}$ are chosen randomly and independently from $\Sigma_j$ with probability distribution such that
    \begin{equation*}
        \mathbb{E}(\mu_j(x) | E_{j-1}) = \mu_{j-1}(x).
    \end{equation*}
For the corresponding sequence of random measures $\{\nu_j\}_{j \in \bN \cap \{ 0\} }$, its limiting measure $\nu$ satisfies all conclusions of Theorem \ref{M1} a.s. 
\end{prop}
To prove Proposition \ref{propGR}, we need the next lemma to combine it with Tao's epsilon removal argument.
\begin{lem}\label{lemGLTao}
    Let $p =6/\alpha $ where $0 < \alpha <1$ and assume that $\nu$ is the measure constructed in Section \ref{secCan} and satisfies \eqref{M1eq2}. For sufficiently large $R>0$, suppose that $\{Q_i\}_{i=1}^M$ is a sparse collection of $R$-cubes in $\bR^2$, which means that their centers $x_1, \cdots, x_M$ are $R^B M^B$-separated from each other for some sufficiently large constant $B$ which will be determined later. For any function $f$ supported on $\cup_{i=1}^M Q_i$, we have
    \begin{equation}\label{eqGLtao}
        \norm{\widehat{f}}_{L^2(d\nu)} \lesssim_{q, \epsilon} R^{\epsilon} \norm{f}_{L^{q}(\bR^2)}
    \end{equation}
    where $1/p + 1/q = 1$.
\end{lem}
\begin{proof}
    Let $f = \sum_{i=1}^M f \mathbf{1}_{Q_i}$ where $\mathbf{1}_{Q_i}$ is a characteristic function on $Q_i$.
    \begin{equation*}
    \begin{split}
        \norm{\widehat{f}}_{L^2(d\nu)}^2 &= \int |\sum_{i=1}^M \widehat{f\mathbf{1}_{Q_i}}|^2 d\nu\\
        &=\int \sum_{i=1}^M | \widehat{f\mathbf{1}_{Q_i}}|^2 d\nu + \int \sum_{i \neq j}  \widehat{f\mathbf{1}_{Q_i}} \overline{ \widehat{f\mathbf{1}_{Q_j}}} d\nu.
    \end{split}
    \end{equation*}
    It follows from Proposition \ref{PropLRresult} that
    \begin{equation}\label{ijeq}
        \begin{split}
            \sum_{i=1}^M \int |\widehat{f\mathbf{1}_{Q_i}}|^2 d\nu &\lesssim_{q,\epsilon} R^\epsilon \sum_{i=1}^M \norm{f\mathbf{1}_{Q_i}}_{L^{q}(\bR^2)}^2 \\
            &\leq R^\epsilon \left( \sum_{i=1}^M \norm{f\mathbf{1}_{Q_i}}_{L^{q}(\bR^2)}^{q} \right)^{2/q}= R^\epsilon \norm{f}_{L^{q}(\bR^2)}^2.
        \end{split}
    \end{equation}
    In the last inequality, we used that $1 \leq q\leq 2$.
    
    For the second term, Young's convolution inequality implies that
    \begin{equation}\label{eqLemij}
        \begin{split}
            \int \sum_{i \neq j} \widehat{f\mathbf{1}_{Q_i}} \overline{ \widehat{f\mathbf{1}_{Q_j}}} d\nu &=\sum_{i \neq j} \iint f\mathbf{1}_{Q_i}(x) \overline{f\mathbf{1}_{Q_j}(y)} \widehat{d\nu}(y-x)dxdy\\
            &\leq\sum_{i \neq j} \norm{f\mathbf{1}_{Q_i}}_{L^{q}(\bR^2)}\norm{f\mathbf{1}_{Q_i}}_{L^{q}(\bR^2)} \norm{\widehat{d\nu}\mathbf{1}_{i,j}}_{L^r(\bR^2)}
        \end{split}
    \end{equation}
    where $\frac{2}{q}+\frac{1}{r}=2$, so that $r= \frac{p}{2} \geq 1$  and 
    \begin{equation*}
        \mathbf{1}_{i,j} = \mathbf{1}_{B^2(2R)}(x-(x_j-x_i))
    \end{equation*}
    where $B^2(2R)$ is a ball of radius $2R$ centered at the origin. Since $x_i$ and $x_j$ are separated at least by $R^B M^B$, \eqref{M1eq2} implies that 
    \begin{equation}\label{dnu1_ij}
        \norm{\widehat{d\nu} \mathbf{1}_{i,j}}_{L^r(\bR^2)} \lesssim_{\epsilon} (R^B M^B)^{-\alpha/2 +\epsilon} R^{2/r}.
    \end{equation}
    We choose $B$ such that 
    \begin{equation*}
        B \geq \frac{8}{\alpha}
    \end{equation*}
    and let $\epsilon = \alpha/4$ in \eqref{dnu1_ij}.
    Since $r \geq 1$ we have, 
    \begin{equation}\label{eqest1_i,j}
        \norm{\widehat{d\nu} \mathbf{1}_{i,j}}_{L^r(\bR^2)} \lesssim (RM)^{-B\alpha/4} R^2 \leq M^{-2}.
    \end{equation}
    By \eqref{eqLemij} and \eqref{eqest1_i,j}, we obtain that
    \begin{equation}\label{ijneq}
    \begin{split}
        \int \sum_{i \neq j} \widehat{f\mathbf{1}_{Q_i}} \overline{\widehat{f\mathbf{1}_{Q_j}}  } d\nu &\lesssim M^{-2} \sum_{i \neq j} \norm{f\mathbf{1}_{Q_i}}_{L^{q}(\bR^2)}\norm{f\mathbf{1}_{Q_j}}_{L^{q}(\bR^2)}\\
        &\leq M^{-1}\left( \sum_{i=1}^M \norm{f\mathbf{1}_{Q_i}}_{L^{q}(\bR^2)}^2 \right)^{1/2} \left( \sum_{j=1}^M \norm{f\mathbf{1}_{Q_j}}_{L^{q}(\bR^2)}^2 \right)^{1/2}\\
        &\leq M^{-1}\norm{f}_{L^{q}(\bR^2)}^2.
    \end{split}
    \end{equation}
    In the last inequality, we used again that $1 \leq q \leq 2$. Combining \eqref{ijeq} and \eqref{ijneq}, we obtain \eqref{eqGLtao}.
\end{proof}
Let us turn to the proof of Proposition \ref{propGR}.
\begin{proof}[Proof of Proposition \ref{propGR}]
    The inequality \eqref{M1eq1} follows from Lemma \ref{CanDimPara}. If we let $n_k^{1/2}  = m_k$, $N_k^{1/2} = M_k$, $\beta_k = |E_k|^{-1} $ and $\alpha_0 = \alpha$, then $\{ \mu_k\}_{k \in \bN \cup \{0\}}$ satisfies all conditions of Proposition \ref{PropFDecay}. Thus, $\widehat{\nu}$ satisfies \eqref{M1eq2} a.s. Now, we can use Tao's epsilon removal argument to prove \eqref{M1eq3}. Lemma \ref{lemGLTao} replaces Lemma 3.2 in \cite{Tao99} and rest of the proof \eqref{M1eq3} is same as in \cite[Theorem 1.2]{Tao99}.
\end{proof}
\section{Proof of Theorem \ref{M2}}\label{secSharp}
The proof is based on Knapp-type example.
\begin{proof}
    Since $\nu$ is supported on $\bP^{d-1}$, there exists a measure $\mu$ on $[0,1]^{d-1}$ such that 
    \begin{equation*}
        \int f(x) d\nu = \int f(x', |x'|^2 )d\mu
    \end{equation*}
   for all $\nu$-measurable function $f$ on $\bR^d$ where $x' \in \bR^{d-1} $. The assumption \eqref{M2eq1} implies that 
    \begin{equation}\label{Mueq1}
        \mu(B(x',r)) \lesssim_{\alpha} r^{\alpha}
    \end{equation}
    for all $x' \in \bR^{d-1}$ and $\forall r >0$. 
    
    We claim that for any $\epsilon >0$, there exist sequences $\{a_i\}_{i \in \bN}$ and $\{r_i\}_{i \in \bN }$ such that $\lim_{i \rightarrow \infty } r_i =0$ and 
    \begin{equation}\label{lbmu}
        r_i^{\alpha+\epsilon} < \mu(B(a_i,r_i)).
    \end{equation}
    The measure $\mu$ is supported on a set of Hausdorff dimension $\alpha $ in $\bR^{d-1}$. By Frostman's lemma (see for example \cite[Theorem 2.7]{Mat15}), $\alpha$ is the supremum of $s$ such that $\mu(B(x',r)) \leq r^s$ holds for all $x \in \bR^{d-1}$ and $r >0$. Towards a contradiction, for given $\epsilon>0$, we assume that $r_0>0$ is the smallest such that $\mu(B(x',r_0)) > r^{\alpha+\epsilon}$ for some $x' \in \bR^{d-1}$. Let $\mu' = r_0^{\alpha+\epsilon} \mu$. Since $\nu$ is a probability measure, $\mu([0,1])^{d-1}=1$. If $r\geq r_0$, then $\mu'(B(x,r)) \leq r_0^{\alpha+\epsilon} \mu([0,1]^{d-1}) \leq r^{\alpha+\epsilon} $. If $r <r_0$, then $\mu'(B(x,r)) \leq \mu(B(x,r)) \leq r^{\alpha+\epsilon} $. Since $\mu'$ is also supported on a set of Hausdorff dimension $\alpha$, it contradicts Frostman's lemma. Therefore, such $r_0$ does not exists and we established the claim \eqref{lbmu}.
    
    Assume that the estimate \eqref{M2eq2} holds for some $p$ and $q$ and let $ h_i(x') = \mathbf{1}_{[0,1]^{d-1}} (r_i^{-1} (x'-a_i))$.
    Then, we have
    \begin{equation}\label{M2pfeq1}
        \left( \int \bigg| \int h_i(x') e(-x'\cdot \xi' - |x'|^2\xi_d ) d\mu(x')\bigg|^p d\xi \right)^{1/p} \lesssim_{p,q}  \left( \int |h_i(x')|^q d\mu \right)^{1/q}    
    \end{equation}
    where $\xi' \in \bR^{d-1}$ and $\xi_d \in \bR$. By \eqref{Mueq1}, we obtain that
    \begin{equation}\label{hupbound}
        \left( \int |h_i(x')|^q d\mu \right)^{1/q} \lesssim_{\alpha} r_i^{\alpha/q}.
    \end{equation}
    By change of variables $\xi' = r_i^{-1}\zeta'$ and $\xi_d = r_i^{-2}\zeta_d$, left-hand side of \eqref{M2pfeq1} becomes
    \begin{equation}\label{hlbound}
    \begin{split}
        &r_i^{-(d+1)/p} \left( \int \bigg|  \int h_i(x') e(-r_i^{-1} (x'-a_i) \cdot(\zeta' + 2 \zeta_d r_i^{-1} a_i ) -r_i^{-2}\zeta_d |x'-a_i|^2)d\mu(x') \bigg|^p d\zeta \right)^{1/p}.
    \end{split}
    \end{equation}
    By considering when $|\zeta'+2\zeta_d r_i^{-1}a_i | \leq 1/100$ and $|\zeta_d| \leq 1/100$, \eqref{lbmu} and \eqref{hlbound} implies that 
    \begin{equation}\label{hlowbound}
    \begin{split}
        &\left( \int \bigg| \int h_i(x') e(-x'\cdot\xi' - |x'|^2\xi_d ) d\mu(x')\bigg|^p d\xi \right)^{1/p}\\
        &\gtrsim r_i^{-(d+1)/p} \bigg|  \int \mathbf{1}_{[0,1]^{d-1}} (r_i^{-1} (x'-a_i)) d\mu(x') \bigg| \gtrsim_{\alpha,\epsilon} r_i^{-(d+1)/p+\alpha +\epsilon }.
    \end{split}
    \end{equation}
    We combine \eqref{M2pfeq1}, \eqref{hupbound} and \eqref{hlowbound}, then for any $r_i$, we get
    \begin{equation*}
        r_i^{-(d+1)/p +\alpha +\epsilon} \lesssim_{\alpha,\epsilon,p,q} r_i^{\alpha/q}.
    \end{equation*}
    Therefore, 
    \begin{equation*}
        \frac{\alpha}{q} \leq -\frac{d+1}{p} + \alpha +\epsilon.
    \end{equation*}
    Since $\epsilon$ is arbitrary, $p \geq q'(d+1)/\alpha$ if the estimate \eqref{M2eq2} holds.
\end{proof}

\section{Appendix}\label{secAppen}

\begin{proof}[Proof of Lemma \ref{DIS_1S_2}]
    First, assume that $S_1$ is a subset of $[-k_0, k_0]^d\cap \bZ^d$. Let $\psi$ be a non-negative smooth function supported on $[0,1]^d$ such that $\psi \geq 1$ on $[1/4, 3/4]^d$. For $x \in [-k,k]^d$,
    \begin{equation*}
        1 \leq \sum_{|n| \lesssim k} |\psi(x-n)|^p + \sum_{|n'| \lesssim 1} \sum_{|n| \lesssim k} |\psi(x-n-n'/2)|^p
    \end{equation*}
    where $n, n' \in \bZ^d$, but $n' \neq 0$. Thus, we have
    \begin{equation*}
    \begin{split}
        &\norm{\sum_{a\in S_1}\sum_{b \in S_{2,a}} c_{a,b} e((ka+b)\cdot x) }_{L^p([0,1]^d)}^p = \int_{[0,k]^d} \bigg|\sum_{a \in S_1} \sum_{b \in S_{2,a}} c_{a,b} e\left( \left(a+ \frac{b}{k} \right)\cdot x\right) \bigg|^p k^{-d}dx\\
        &\leq \sum_{|n| \lesssim k} k^{-d} \int_{[0,1]^d} \bigg| \sum_{a \in S_1} \sum_{b \in S_{2,a}} c_{a,b} e(a\cdot x) e\left( \frac{b}{k}\cdot (x+n)\right) \psi(x) \bigg|^p dx\\
        &+\sum_{|n'| \lesssim 1} \sum_{|n| \lesssim k}k^{-d} \int_{[0,1]^d} \bigg| \sum_{a \in S_1} \sum_{b \in S_{2,a}} c_{a,b} e(a\cdot (x+n'/2)) e\left( \frac{b}{k}\cdot (x+n+n'/2)\right) \psi(x) \bigg|^p dx.
    \end{split}
    \end{equation*}
    We will show that the first term is bounded by a constant multiple of
    \begin{equation}\label{C_1C_2bound}
         C_1C_2 \left( \sum_{a \in S_1} \sum_{ b\in S_{2,a} } |c_{a,b}|^2 \right)^{p/2}.
    \end{equation}
    By the same argument, the second term will be also bounded by a constant multiple of \eqref{C_1C_2bound}.
    
    Since $\psi$ is supported on $[0,1]^d$, we have
    \begin{equation*}
        \psi(x) e(b\cdot x/k) = \sum_{m \in \bZ^d} \widehat{\psi}_b(m) e(m\cdot x)
    \end{equation*}
    where $\widehat{\psi}_b(m)$ is the $m$-th Fourier coefficient of $\psi(x) e(b\cdot x/k)$. Since $\sum_m |\widehat{\psi}_b(m)| \lesssim 1$ and $S_1$ is a subset of $[-k_0,k_0]^d$, we obtain
    \begin{equation}\label{appeneq1}
    \begin{split}
        & \left(\sum_{|n| \lesssim k} k^{-d} \int_{[0,1]^d} \bigg| \sum_{a \in S_1} \sum_{b \in S_{2,a}} c_{a,b} e(a\cdot x) e\left( \frac{b}{k}\cdot (x+n) \right) \psi(x) \bigg|^p dx \right)^{1/p}\\
        &\leq \sum_{m \in \bZ^d} \left( \sum_{|n| \lesssim k} k^{-d} \int_{[0,1]^d}  \bigg| \sum_{a \in S_1} \sum_{b \in S_{2,a}} c_{a,b} e \left( \frac{b\cdot n}{k} \right) \widehat{\psi}_b(m) e((a+m)\cdot x)  \bigg|^p dx \right)^{1/p}\\
        &\leq \sum_{m \in \bZ^d} \left( \sum_{|n| \lesssim k} k^{-d} \int_{[0,1]^d}  \bigg| \sum_{a \in S_1}  \left( \sum_{b \in S_{2,a}} c_{a,b} e \left( \frac{b\cdot n}{k} \right) \widehat{\psi}_b(m) \right) e(a\cdot x)  \bigg|^p dx \right)^{1/p}\\
        &\leq C_1 \sum_{m \in \bZ^d} \left( \sum_{|n| \lesssim k} k^{-d} \left( \sum_{a\in S_1} \bigg| \sum_{b\in S_{2,a}}c_{a,b} e\left(\frac{b\cdot n
        }{k} \right) \widehat{\psi}_b(m) \bigg|^2 \right)^{p/2} \right)^{1/p}.
    \end{split}
    \end{equation}
    In the last inequality, we used \eqref{DIS_1S_2eq1}.
    
    First, we consider when $|m| \leq 10$. Since $p \geq 2$ and $|\psi|\lesssim 1$, we get
    \begin{equation}\label{m_small}
        \begin{split}
            &\left( \sum_{|n| \lesssim k} k^{-d} \left( \sum_{a\in S_1} \bigg| \sum_{b\in S_{2,a}}c_{a,b} e\left(\frac{b\cdot n}{k} \right) \widehat{\psi}_b(m) \bigg|^2 \right)^{p/2} \right)^{1/p}\\
            &\leq \left( \sum_{a \in S_1} \left( \sum_{|n| \lesssim k} k^{-d} \bigg| \sum_{b \in S_{2,a}} c_{a,b} e\left(\frac{b\cdot n}{k} \right) \int_{[0,1]^d} \psi(x) e\left( \left( \frac{b}{k}-m \right)\cdot x \right)dx \bigg|^p \right)^{2/p} \right)^{1/2}\\
            &\lesssim \left( \sum_{a \in S_1} \left( k^{-d} \sum_{|n| \lesssim k} \int_{[0,1]^d} \bigg| \sum_{b \in S_{2,a}} c_{a,b} e\left( \frac{b}{k}\cdot (x+n) \right) \bigg|^p dx \right)^{2/p} \right)^{1/2}\\
            &\lesssim \left( \sum_{a \in S_1} \norm{\sum_{b \in S_{2,a}} c_{a,b} e(b\cdot x)}_{L^p([0,1]^d)}^2 \right)^{1/2} \leq  C_2 \left( \sum_{a \in S_1} \sum_{b \in S_{2,a}} |c_{a,b}|^2 \right)^{1/2}.
        \end{split}
    \end{equation}
    In the case when $|m| \geq 10$, since $\psi$ is supported on $[0,1]^d$ and decreases rapidly, we can use the principle of stationary phase. Let $(b/k-m)_i$ denote the $i$-th coordinate of $(b/k-m)$ and denote $\nabla_{b,k,m}$ by 
    \begin{equation*}
        \nabla_{b,k,m} = \sum_{i=1}^d \bigg| \frac{b}{k}-m \bigg|^{-2} \left( \frac{b}{k}-m \right)_i \frac{\partial}{\partial x_i}.
    \end{equation*}
    Note that 
    \begin{equation*}
        \widehat{\psi}_b(m)  = \int_{[0,1]^d} (\nabla_{b,k,m}^2 \psi(x)) e\left(\left(\frac{b}{k}-m\right) x \right) dx.
    \end{equation*}
    If we let
    \begin{equation*}
        \tilde{c}_{a,b,i,j,m}  = c_{a,b}\bigg|\frac{b}{k}-m \bigg|^{-4} \left(\frac{b}{k}-m \right)_i \left(\frac{b}{k}-m \right)_j,
    \end{equation*}
    it is well defined since $|b| \leq k$ and $|m| \geq 10$. Then we obtain 
    \begin{equation*}
        \begin{split}
            \bigg|\sum_{b \in S_{2,a}} c_{a,b} e\left(\frac{b \cdot n}{k} \right)\widehat{\psi}_b(m)\bigg| \leq \sum_{i,j=1}^d \int_{[0,1]^d} \bigg| \sum_{b \in S_{2,a}} \tilde{c}_{a,b,i,j,m} e\left( \frac{b}{k}\cdot (x+n) \right) \partial_{ij} \psi(x)\bigg|dx.
        \end{split}
    \end{equation*}
    Since $|\partial_{i,j} \psi| \lesssim 1$ uniformly in $i$ and $j$, we obtain that
    \begin{equation*}
        \begin{split}
            &\left( \sum_{|n| \lesssim k} k^{-d} \left( \sum_{a\in S_1} \bigg| \sum_{b\in S_{2,a}}c_{a,b} e\left(\frac{b\cdot n}{k} \right) \widehat{\psi}_b(m) \bigg|^2 \right)^{p/2} \right)^{1/p}\\
            &\lesssim \sum_{i,j=1}^d \left( \sum_{a \in S_1} \left( k^{-d} \sum_{|n| \lesssim k} \int_{[0,1]^d} \bigg| \sum_{b \in S_{2,a}} \tilde{c}_{a,b,i,j,m} e\left( \frac{b}{k}\cdot (x+n) \right) \bigg|^p dx \right)^{2/p} \right)^{1/2}\\
            &\lesssim \sum_{i,j=1}^d \left( \sum_{a \in S_1} \norm{\sum_{b \in S_{2,a}} \tilde{c}_{a,b,i,j,m} e(b\cdot x)}_{L^p([0,1]^d)}^2 \right)^{1/2}\\
            &\leq C_2 \sum_{i,j=1}^d \left( \sum_{a \in S_1} \sum_{b \in S_{2,a}} |\tilde{c}_{a,b,i,j,m}|^2 \right)^{1/2} \lesssim_d C_2|m|^{-2} \left( \sum_{a\in S_1} \sum_{b \in S_{2,a}} |c_{ab}|^2\right)^{1/2}.
        \end{split}
    \end{equation*}
    In the last inequality, we used that $|b/k-m | \approx |m|$ since $|m| \geq 10$. Therefore,
    \begin{equation}\label{appeneq2}
        \begin{split}
            &\sum_{|m| \geq 10} \left( \sum_{|n| \lesssim k} k^{-d} \left( \sum_{a\in S_1} \bigg| \sum_{b\in S_{2,a}}c_{a,b} e\left(\frac{b\cdot n}{k} \right) \widehat{\psi}_b(m) \bigg|^2 \right)^{p/2} \right)^{1/p}\\
            &\lesssim C_2 \sum_{|m| \geq 10} |m|^{-2} \left( \sum_{a \in S_1} \sum_{b \in S_{2,a}} |c_{a,b}|^2 \right)^{1/2} \\
            &\lesssim C_2 \left( \sum_{a \in S_1} \sum_{b \in S_{2,a}} |c_{a,b}|^2 \right)^{1/2}.
        \end{split}
    \end{equation}
    By \eqref{appeneq1}, \eqref{m_small} and \eqref{appeneq2}, it follows that
    \begin{equation*}
        \begin{split}
            &\left(\sum_{|n| \lesssim k} k^{-d} \int_{[0,1]^d} \bigg| \sum_{a \in S_1} \sum_{b \in S_{2,a}} c_{a,b} e(a\cdot x) e\left( \frac{b}{k}\cdot(x+n) \right) \psi(x) \bigg|^p dx \right)^{1/p}\\
            &\lesssim C_1 C_2 \left( \sum_{a \in S_1} \sum_{b \in S_{2,a}} |c_{a,b}|^2 \right)^{1/2}.
        \end{split}
    \end{equation*}
    
    We proved Lemma \ref{DIS_1S_2} when $S_1$ is a subset of $[-k_0,k_0]^d \cap \bZ^d$. The estimate holds for arbitrary $k_0$. Thus, letting $k_0 \rightarrow \infty$ finishes the proof.
\end{proof}
\begin{proof}[Proof of Lemma \ref{DIequiv}]
    The inequality $K_p(E_j) \lesssim_p D_p(E_j)$ is well known (see for example \cite[Theorem 13.1]{Dem20}). Thus, it suffices to prove the converse. Let $R$ be a positive large number and $S$ be a subset of $ [-R,R]^d \cap \bZ^d$ such that 
    \begin{equation}\label{AppenDR}
        \norm{\sum_{a \in S} c_a e(a \cdot x)}_{L^p([0,1]^d)} \leq C \left( \sum_{a\in S}|c_a|^2\right)^{1/2}
    \end{equation}
    and let  
    \begin{equation*}
        \widehat{f}_a(\xi) = \widehat{f}(\xi) \mathbf{1}_{[0,1]^d} (R\xi -a).
    \end{equation*}
    Let us consider a Schwartz function $\eta$ such that $|\eta| \geq 1$ on $[-1,1]^d$ and $\widehat{\eta}$ is supported on $[0,1]^d$. Recall that $B^d(R)$ is a cube of side length $R$ in $ \bR^d$ centered at the origin. We will show that for any $f$ such that $f=\sum_{a \in S} f_a$,
    \begin{equation}\label{AppenDec}
        \norm{f}_{L^p(B^d(R))} \lesssim C \left( \sum_{a \in S} \norm{f_a}_{L^p(\eta_R)}^2 \right)^{1/2}
    \end{equation}
    where $\eta_R(x) = \eta(x/R)$ and $C$ is the constant in \eqref{AppenDR}. Then \eqref{AppenDec} implies that
    \begin{equation}\label{AppenDec2}
        \norm{f}_{L^p(\bR^d)} \lesssim C \left( \sum_{a \in S} \norm{f_a}_{L^p(\bR^d)}^2 \right)^{1/2}.
    \end{equation}
    For the proof of \eqref{AppenDec2} from \eqref{AppenDec}, see \cite[Proposition 9.15]{Dem20} and \cite[Chapter 4.1]{Zli19}. Thus, Lemma \ref{DIequiv} easily follows.
    
    Now we prove \eqref{AppenDec}. We have
    \begin{equation}\label{A_LpBR}
    \begin{split}
        \norm{f}_{L^p(B^d(R))} &\leq \norm{f\eta_R}_{L^p(B^d(R))}\\
        &= \norm{\sum_{a \in S} \int \widehat{f}_a\ast \widehat{\eta}_R(\xi) e(-x\cdot \xi) d\xi  }_{L^p(B^d(R))}.
    \end{split}
    \end{equation}
    The function $\widehat{f}_a\ast \widehat{\eta}_R$ is supported on a cube of side length $O(R^{-1})$ centered at $a/R$. Therefore, by change of variables $\xi = R^{-1} \zeta$ and $R^{-1}x = z$, \eqref{A_LpBR} is
    \begin{equation*}
        \begin{split}
            &\leq R^{-d+d/p} \sum_{|c| \lesssim 1} \int_{[0,1]^d} \norm{ \sum_{a\in S} \widehat{f}_a \ast \widehat{\eta}_R (R^{-1}(\zeta+a+c)) e(-z\cdot(\zeta+a+c))}_{L^p(B^d(1))} d\zeta\\
            &\lesssim R^{-d+d/p} \sum_{|c| \lesssim 1} \int_{[0,1]^d} \norm{ \sum_{a\in S} \widehat{f}_a \ast \widehat{\eta}_R (R^{-1}(\zeta+a+c)) e(-z\cdot a)}_{L^p(B^d(1))} d\zeta.
        \end{split}
    \end{equation*}
    By \eqref{AppenDR}, we obtain that
    \begin{equation*}
        \begin{split}
            \norm{f}_{L^p(B^d(R))}&\lesssim  C R^{-d+d/p}\sum_{|c| \lesssim 1} \int_{[0,1]^d} \left(\sum_{a\in S} |\widehat{f}_a\ast \widehat{\eta}_R(R^{-1}(\zeta+a+c))|^2\right)^{1/2} d\zeta\\
            & \lesssim C R^{-d+d/p} \sum_{|c| \lesssim 1} \left(  \sum_{a\in S} \int_{[0,1]^d} |\widehat{f}_a\ast \widehat{\eta}_R(R^{-1}(\zeta+a+c))|^2 d\zeta \right)^{1/2}\\
            & \lesssim C R^{-d/2+d/p} \left( \sum_{a \in S} \norm{\widehat{f}_a \ast \widehat{\eta}_R(\xi)}_{L^2(\bR^d)}^2 \right)^{1/2}\\
            &\leq C R^{-d/2 +d/p} \left( \sum_{a \in S} \norm{f_a \eta_R}_{L^2(\bR^d)}^2 \right)^{1/2}\\
            &\leq C R^{-d/2 +d/p}\left( \sum_{a \in S} \norm{f_a}_{L^p(\eta_R)}^2 \norm{\eta_R}_{L^1(\bR^d)}^{1-2/p} \right)^{1/2}.
        \end{split}
    \end{equation*}
    Since $\norm{\eta_R}_{L^1} \approx R^d$, this completes the proof of \eqref{AppenDec}.

\end{proof}


\begin{thebibliography}{10}

\bibitem{BS11}
Jong-Guk Bak and Andreas Seeger.
\newblock Extensions of the {S}tein-{T}omas theorem.
\newblock {\em Math. Res. Lett.}, 18(4):767--781, 2011.

\bibitem{BP61}
A.~Benedek and R.~Panzone.
\newblock The space {$L^{p}$}, with mixed norm.
\newblock {\em Duke Math. J.}, 28:301--324, 1961.

\bibitem{Bour89}
J.~Bourgain.
\newblock Bounded orthogonal systems and the {$\Lambda(p)$}-set problem.
\newblock {\em Acta Math.}, 162(3-4):227--245, 1989.

\bibitem{CDGJLM22}
Alan Chang, Jaume de~Dios~Pont, Rachel Greenfeld, Asgar Jamneshan, Zane~Kun Li,
  and Jos\'{e} Madrid.
\newblock Decoupling for fractal subsets of the parabola.
\newblock {\em Math. Z.}, 301(2):1851--1879, 2022.

\bibitem{Xchen14}
Xianghong Chen.
\newblock A {F}ourier restriction theorem based on convolution powers.
\newblock {\em Proc. Amer. Math. Soc.}, 142(11):3897--3901, 2014.

\bibitem{Xchen16}
Xianghong Chen.
\newblock Sets of {S}alem type and sharpness of the {$L^2$}-{F}ourier
  restriction theorem.
\newblock {\em Trans. Amer. Math. Soc.}, 368(3):1959--1977, 2016.

\bibitem{CS17}
Xianghong Chen and Andreas Seeger.
\newblock Convolution powers of {S}alem measures with applications.
\newblock {\em Canad. J. Math.}, 69(2):284--320, 2017.

\bibitem{Dem20}
Ciprian Demeter.
\newblock {\em Fourier restriction, decoupling, and applications}, volume 184
  of {\em Cambridge Studies in Advanced Mathematics}.
\newblock Cambridge University Press, Cambridge, 2020.

\bibitem{Dur19}
Rick Durrett.
\newblock {\em Probability---theory and examples}, volume~49 of {\em Cambridge
  Series in Statistical and Probabilistic Mathematics}.
\newblock Cambridge University Press, Cambridge, 2019.
\newblock Fifth edition of [ MR1068527].

\bibitem{HL13}
Kyle Hambrook and Izabella \L~aba.
\newblock On the sharpness of {M}ockenhaupt's restriction theorem.
\newblock {\em Geom. Funct. Anal.}, 23(4):1262--1277, 2013.

\bibitem{HL16}
Kyle Hambrook and Izabella \L~aba.
\newblock Sharpness of the {M}ockenhaupt-{M}itsis-{B}ak-{S}eeger restriction
  theorem in higher dimensions.
\newblock {\em Bull. Lond. Math. Soc.}, 48(5):757--770, 2016.

\bibitem{Ho63}
Wassily Hoeffding.
\newblock Probability inequalities for sums of bounded random variables.
\newblock {\em J. Amer. Statist. Assoc.}, 58:13--30, 1963.

\bibitem{Kauf76}
R.~Kaufman.
\newblock Random measures on planar curves.
\newblock {\em Ark. Mat.}, 14(2):245--250, 1976.

\bibitem{Kor11}
Thomas~William K\"{o}rner.
\newblock Hausdorff and {F}ourier dimension.
\newblock {\em Studia Math.}, 206(1):37--50, 2011.

\bibitem{LW18}
Izabella \L~aba and Hong Wang.
\newblock Decoupling and near-optimal restriction estimates for {C}antor sets.
\newblock {\em Int. Math. Res. Not. IMRN}, 2018(9):2944--2966, 2018.

\bibitem{Zli19}
Zane~Kun Li.
\newblock {\em Decoupling for the parabola and connections to efficient
  congruencing}.
\newblock University of California, Los Angeles, 2019.

\bibitem{Ma04}
Pertti Mattila.
\newblock Hausdorff dimension, projections, and the {F}ourier transform.
\newblock {\em Publ. Mat.}, 48(1):3--48, 2004.

\bibitem{Mat15}
Pertti Mattila.
\newblock {\em Fourier analysis and {H}ausdorff dimension}, volume 150 of {\em
  Cambridge Studies in Advanced Mathematics}.
\newblock Cambridge University Press, Cambridge, 2015.

\bibitem{Mit02}
Themis Mitsis.
\newblock A {S}tein-{T}omas restriction theorem for general measures.
\newblock {\em Publ. Math. Debrecen}, 60(1-2):89--99, 2002.

\bibitem{Mi02}
Themis Mitsis.
\newblock A {S}tein-{T}omas restriction theorem for general measures.
\newblock {\em Publ. Math. Debrecen}, 60(1-2):89--99, 2002.

\bibitem{Mo00}
G.~Mockenhaupt.
\newblock Salem sets and restriction properties of {F}ourier transforms.
\newblock {\em Geom. Funct. Anal.}, 10(6):1579--1587, 2000.

\bibitem{ShSu17}
Pablo Shmerkin and Ville Suomala.
\newblock A class of random {C}antor measures, with applications.
\newblock In {\em Recent developments in fractals and related fields}, Trends
  Math., pages 233--260. Birkh\"{a}user/Springer, Cham, 2017.

\bibitem{ShSu18}
Pablo Shmerkin and Ville Suomala.
\newblock Spatially independent martingales, intersections, and applications.
\newblock {\em Mem. Amer. Math. Soc.}, 251(1195):v+102, 2018.

\bibitem{StSh11}
Elias~M. Stein and Rami Shakarchi.
\newblock {\em Functional analysis}, volume~4 of {\em Princeton Lectures in
  Analysis}.
\newblock Princeton University Press, Princeton, NJ, 2011.
\newblock Introduction to further topics in analysis.

\bibitem{Tal95}
Michel Talagrand.
\newblock Sections of smooth convex bodies via majorizing measures.
\newblock {\em Acta Math.}, 175(2):273--300, 1995.

\bibitem{Tal21}
Michel Talagrand.
\newblock Upper and lower bounds for stochastic processes. decomposition
  theorems.
\newblock {\em Ergebnisse der Mathematik und ihrer Grenzgebiete}, 60, 2021.

\bibitem{Tao99}
Terence Tao.
\newblock The {B}ochner-{R}iesz conjecture implies the restriction conjecture.
\newblock {\em Duke Math. J.}, 96(2):363--375, 1999.

\bibitem{Wo03}
Thomas~H. Wolff.
\newblock {\em Lectures on harmonic analysis}, volume~29 of {\em University
  Lecture Series}.
\newblock American Mathematical Society, Providence, RI, 2003.
\newblock With a foreword by Charles Fefferman and a preface by Izabella \L
  aba, Edited by \L aba and Carol Shubin.

\end{thebibliography}
\end{document}